\documentclass{amsart}

\usepackage{homaloidal}

\title{Homaloidal Polynomials and Gaussian Models of Maximum Likelihood Degree One}
\date{October 15, 2024}

\begin{document}

\begin{abstract}
We study the Gaussian statistical models whose log-likelihood function has a unique complex critical point, i.e., has maximum likelihood degree one. We exploit the connection developed by Améndola et. al. between the models having maximum likelihood degree one and homaloidal polynomials. 
We study the spanning tree generating function of a graph and show this polynomial is homaloidal when the graph is chordal. When the graph is a cycle on $n$ vertices, $n \geq 4$, we prove the polynomial is not homaloidal, and show that the maximum likelihood degree of the resulting model is the $(n-1)$st Eulerian number. These results support our conjecture that the spanning tree generating function is a homaloidal polynomial if and only if the graph is chordal.
We also provide an algebraic formulation for the defining equations of these models. Using existing results, we provide a computational study on constructing new families of homaloidal polynomials. In the end, we analyze the symmetric determinantal representation of such polynomials and provide an upper bound on the size of the matrices involved. 
\end{abstract}

\maketitle


\section{Introduction}

Let $\mathcal{M}$ be a semi-algebraic subset of $\PD_m$, the cone of real positive-definite $m\times m$ matrices. In this paper, we consider $\mathcal{M}$ as an affine-linear Gaussian concentration model, i.e. $\mathcal{M}$ is the set of concentration (inverse covariance) matrices corresponding to a subset of mean-zero Gaussian distributions such that these matrices lie on an affine subspace of $m\times m$ symmetric matrices. For an independent identically distributed sample $Y_1,\ldots,Y_n\in \mathbb{R}^m$, the \textit{Gaussian log-likelihood function} on $\mathcal{M}$ is defined as:
\begin{align*}
    \ell_S : \mathcal{M} \to \mathbb{R}, \qquad K \mapsto \log \det(K) - \tr(KS),
\end{align*}
where $S$ is the sample covariance matrix $S = \frac{1}{n} \sum_{i=1}^n Y_i Y_i^T$. 
A \textit{maximum likelihood estimate} (\textit{MLE}) is any concentration matrix in $\mathcal{M}$ that globally maximizes $\ell_S$. It is an estimate for the concentration matrix of the distribution from which the sample $Y_1, \ldots, Y_n$ arises. 
The \textit{maximum likelihood degree (ML degree)} of a Gaussian model $\mathcal{M}$ is the number of complex critical points of the log-likelihood function over $\overline{\mathcal{M}}$, the Zariski closure of the complexification of $\mathcal{M}$, given generic data. 
Thus, the ML degree is an algebraic measure of the complexity of the MLE problem. 
An open problem is to determine models $\mathcal{M}$ with ML degree one. For instance, in the case of Gaussian graphical models it is known that the ML degree is one when the corresponding graph is chordal~\cite{sturmfelsUhler}. Recent progress towards solving this problem was made in \cite{améndola2023differential}, where the authors start by extending the definition of the Gaussian log-likelihood function to a coordinate-free context using projective varieties:

\begin{definition}{\cite[Definition 2.1]{améndola2023differential}} \label{def: coordinate-free ML degree}
Let $\mathcal{L}$ be a finite-dimensional $\mathbb{C}$-vector space and let $X\subseteq \mathbb{P} (\mathcal{L})$ be a projective variety. Fix a homogeneous polynomial $P$ on $\mathcal{L}$. For a vector $u$ in the dual space $\mathcal{L}^\ast$, and $s\in \mathcal{L}\setminus V(P)$, the \textit{Gaussian log-likelihood} is 
\begin{align*}
    \ell_{P,u}(s) \coloneqq \log(P(s)) - u(s).
\end{align*}
The \textit{(Gaussian) ML degree} of $X$ with respect to $P$, denoted by $\mld_P(X)$, is the number of critical points of $\ell_{P,u}$ over the domain $X^{{\rm sm}}\setminus V(P)$ for general $u$, where $X^{{\rm sm}}$ is the smooth locus of $X$.
\end{definition}

This allows for a  proof of the equivalence between the usual and coordinate-free ML degrees by considering $\mathcal{L}$ to be the space of symmetric matrices.

\begin{prop}{\cite[Proposition 2.3]{améndola2023differential}} 
    \label{prop: homaloidal to det}
    Let $\det$ denote the determinant function on $\sym{m}$, the set of complex symmetric $m\times m$ matrices.
    \begin{enumerate}[i.]
        \item Let $X$ be a subvariety of $\sym{m}$. Then $\mld_{\det}(X)$ is the usual Gaussian ML degree of $X$.
        \item Conversely, let $\mathcal{L}$ be a finite-dimensional $\mathbb{C}$-vector space, $X\subseteq \mathcal{L}$ a variety, and $P$ a polynomial on $\mathcal{L}$. Then there exists an affine-linear embedding $\mathcal{A} : \mathcal{L} \to \sym{m}$ such that $\mld_P(X)= \mld_{\det} (\mathcal{A}(X))$.
    \end{enumerate}
\end{prop}

In this paper, we are specifically interested in using \cite[Proposition 4.2]{améndola2023differential}, where the authors find a necessary and sufficient condition on $P$ such that $\mld_P(X)$ is one. This condition on $P$ comes from a classical construction in algebraic geometry called \textit{homaloidal} polynomials. 

\begin{definition}
    \label{defn: homaloidal}
    A homogeneous polynomial $P(x_0,x_1,\ldots,x_n)$ is said to be homaloidal if the rational map 
    \[
    (t_0,t_1,\ldots, t_n) \rightarrow \left( \frac{\partial P}{\partial x_0} (t), \frac{\partial P}{\partial x_1} (t), \ldots , \frac{\partial P}{\partial x_n} (t) \right)
    \]
    is birational. A hypersurface is homaloidal if its defining polynomial is homaloidal.
\end{definition}

\begin{prop}{\cite[Proposition 4.2]{améndola2023differential}} 
    \label{prop: homaloidal iff mld 1}
    Let $P$ be a homogeneous polynomial on $\mathcal{L}$. Then $\mld_P(\mathbb{P}(\mathcal{L}))$ is one if and only if $P$ is a homaloidal polynomial.     
\end{prop}

A major part of this paper focuses on developing new ways of constructing such polynomials. In \Cref{sec: spanning tree models}, we specifically study the polynomials obtained as the spanning tree generating function of a connected graph $G$. These functions correspond to linear concentration models $\mathcal{M}(G, k)$ that arise from the Laplacian of $G$ after deleting the row and column corresponding to vertex $k$. These models generalize Gaussian graphical models; in statistics, these models are known as H\"usler-Reiss models \cite{MR4630937, röttger2023parametric}. For a fixed graph $G$, each principal submatrix of the Laplacian has the same determinant by the matrix-tree theorem. Thus, each model that arises in this way from $G$ has the same ML degree, and the same determinant function, which is the spanning tree polynomial. In \Cref{thm: chordal implies ML degree one}, we exploit this fact to provide a new proof that $\mathcal{M}(G, k)$ is ML degree one if $G$ is chordal \cite[Proposition 4.4]{hentschel2023statistical}.

\begin{reptheorem}{thm: chordal implies ML degree one}
    If $G$ is chordal, then $\mathcal{M}(G, k)$ has ML degree one for any $k$.
\end{reptheorem}

Thus, if $G$ is a chordal graph, the spanning tree polynomial is homaloidal. Furthermore, we explicitly compute the ML degree of the models for cycle graphs. We find that the ML degree of the models arising from a cycle on $n$ vertices is equal to the $(n-1)$st Eulerian number. 

\begin{reptheorem}{thm: cycle ml degree}
    For all $k$, the ML degree of $\mathcal{M}(C_n, k)$ is equal to the $(n-1)$st Eulerian number, which is the number of subsets of $[n-1]$ with at least 2 elements.
\end{reptheorem}

This provides evidence for the following conjecture:

\begin{repconjecture}{conj: ml deg one iff chordal}
    $\mathcal{M}(G, k)$ has ML degree one if and only if $G$ is chordal.
\end{repconjecture}

Finally, we describe the defining equations of the corresponding covariance model, $\mathcal{M}(G,k)^{-1}$. Regardless of whether $G$ is chordal, the vanishing ideal of the model satisfies the following relation:

\begin{reptheorem}{thm: covariance ideal equations}
    Let $G$ be a graph with $n$ vertices and $G'$ be the subgraph of $G$ obtained after removing the vertex $k$. Let $\Sigma$ be a symmetric matrix of variables indexed by the vertices of $G^\prime$, and $\mathcal{R}_k$ be the set of polynomials in the entries of $\Sigma$ given by
    \[
    \mathcal{R}_k=\{\det \Sigma_{\widehat{i}, \widehat{j}} \mathrel{\big|} l_{ij}=0, i,j\neq k \} \cup \left\{ \sum_i (-1)^{i+j} \det \Sigma_{\widehat{i}, \widehat{j}} \mathrel{\Big|} \, l_{jk}=0\right\},
    \] 
    where 
    $\Sigma_{\widehat{i}, \widehat{j}}$ is the submatrix of $\Sigma$ obtained after removing the $i$th row and $j$th column, and $l_{ij}$ is the $(i,j)$ entry of the Laplacian of $G$.
    Then the vanishing ideal of the model $I(\mathcal{M}(G,k)^{-1})$ is equal to $\langle \mathcal{R}_k \rangle : \det(\Sigma)^\infty$.
\end{reptheorem}

Homaloidal polynomials have been the object of interest in classical algebraic geometry \cite{cremonaTransformations, ciliberto2007homaloidalHA, homaloidal-dets, brunoHomaloidal}, where some work has been done on constructing and classifying them. For instance, in \cite{cremonaTransformations}, the author proved that if $P$ is a homaloidal polynomial in three variables without multiple factors, then the variety of $P$ is either a non-singular conic, or the union of three non-concurrent lines or the union of a conic and its tangent. We study these results in more detail and construct more examples of homaloidal polynomials in \Cref{sec: examples of homaloidals}.

Given a homaloidal polynomial $P$, \cite[Proposition 2.3]{améndola2023differential} proves that expressing the polynomial as the determinant of a matrix with affine-linear entries leads to finding the corresponding affine-linear concentration model with ML degree one. Therefore, after computing examples of homaloidal polynomials in \Cref{sec: examples of homaloidals}, our next step is to find a symmetric determinantal representation for these polynomials in \Cref{sec: SDR for homaloidals}.
\begin{definition}
    Let $\mathbb{F}$ be a field, and $P \in \mathbb{F}[x_1,\ldots,x_n]$ be a polynomial. A \textit{symmetric determinantal representation (SDR)} for $P$ is an expression $P = \det\left(M(x_1,\ldots,x_n)\right)$, where
    \begin{align*}
       M(x_1,\ldots,x_n) = A_0 + \sum_{i=1}^n x_i A_i
    \end{align*}
    such that $A_0,\ldots,A_n$ are symmetric matrices in $\mathbb{F}^{k\times k}$ for some $k\in \mathbb{N}$. The matrix $M(x_1,\ldots,x_n)$ is called an \textit{SDR matrix} of $P$.
\end{definition}
It was first proven in 2006~\cite{HELTON2006105} that any real polynomial has an SDR. In 2021, Stefan and Welters~\cite{SDRshortProof2021} came up with a simpler proof to extend this result to polynomials with coefficients in an arbitrary field $\mathbb{F}$ whose characteristic is not equal to $2$. Motivated by their proof, in \Cref{sec: SDR for homaloidals}, we find an upper bound on the size of the smallest SDR matrix possible for a given polynomial $P\in \mathbb{F}[x_1,\ldots,x_n]$ when $\mathbb{F}\in \{\mathbb{R},\mathbb{C}\}$ by proposing an explicit method for finding the SDR matrix of $P$. In the same section, we briefly discuss whether $M(x_1,\ldots,x_n)$ can be a positive definite matrix for some real values of $x_1,\ldots,x_n$. For the SDR matrix $M(x_1,\ldots,x_n)$ resulting from our construction, we give positive and negative answers to this question in specific cases, and defer the study of other cases to future work.

All code associated with this paper, including specific computations mentioned in the examples, is available in the following GitHub repository
\begin{center}
    \href{https://github.com/shelbycox/Homaloidal}{https://github.com/shelbycox/Homaloidal}
\end{center}

\section{Spanning Tree Models}
\label{sec: spanning tree models}
Let $G = ([n],E)$ be a weighted simple connected graph. The models we consider consist of the principal submatrices of the Laplacian matrix, $L(G)$, viewed as concentration matrices for a mean-zero Gaussian distribution. All graphical concentration models arise in this way, however, this class also includes many new models.
\begin{notation}
    Let $M$ be an $m\times n$ matrix, $I\subseteq [m]$, and $J\subseteq [n]$. Then 
    \begin{enumerate}[i.]
        \item $M_{\widehat{I},\widehat{J}}$ denotes the submatrix of $M$ with rows $I$ and columns $J$ omitted,
        \item $M_{I,J}$ denotes the submatrix of $M$ with only rows $I$ and columns $J$ kept,
        \item $M_{I,\widehat{J}}$ denotes the submatrix of $M$ with only rows $I$ kept, and columns $J$ omitted, and 
        \item $M_{\widehat{I},J}$ denotes the submatrix of $M$ with rows $I$ omitted, and only columns $J$ kept.
    \end{enumerate}
\end{notation}

\begin{definition}
    Let $G=([n],E)$ be a weighted simple graph, and let $k$ be any vertex of $G$. The spanning tree concentration model of $G$ at $k$ is
    \begin{equation*}
        \mathcal{M}(G, k) \coloneqq \left\{ L(G)_{\widehat{k}, \widehat{k}} (\bfx_E) \mid \bfx_E \in \CC^{|E|} \right\}.
    \end{equation*}
    The corresponding covariance model, $\mathcal{M}(G,k)^{-1}$, is defined as follows:
    \begin{align*}
        \mathcal{M}(G,k)^{-1} \coloneqq \overline{\left\{\Sigma \in {\sym{n-1}} \mid \Sigma^{-1} \in \mathcal{M}(G,k) \right\}}.
    \end{align*}
\end{definition}

\begin{example}\label{example:chordal laplacian}
    Let $G$ be the graph on the left of \Cref{fig: first example graph}. The Laplacian of $G$, and the models obtained from deleting its first row/column (resulting in $\mathcal{M}(G, 1)$) or second row/column (resulting in $\mathcal{M}(G, 2)$) are described below. Since $G$ is chordal, \Cref{thm: chordal implies ML degree one} implies that both of these models have ML degree one. The first of the two models, $\mathcal{M}(G, 1)$, is isomorphic to the Gaussian graphical model on the path graph on the right of~\Cref{fig: first example graph}.

    \begin{figure}
        \centering
        \begin{minipage}{.5\textwidth}
            \centering
            \includegraphics[scale=1]{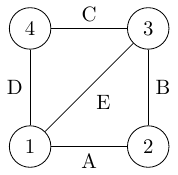}
        \end{minipage}%
        \begin{minipage}{.5\textwidth}
            \centering
            \includegraphics[scale=1]{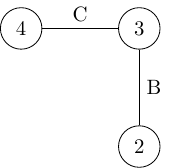}
        \end{minipage}
        \caption{(Left) A chordal graph on four vertices, and (right) the path graph obtained by removing vertex 1 from the graph on the left. The model $\mathcal{M}(G, 1)$ is the Gaussian graphical model corresponding to the graph on the right.}
        \label{fig: first example graph}
    \end{figure}

    \[
    L(G) = 
    \bordermatrix{
        & \textcolor{blue}{\mathbf{1}} & \textcolor{blue}{\mathbf{2}} & \textcolor{blue}{\mathbf{3}} & \textcolor{blue}{\mathbf{4}} \cr
        \textcolor{blue}{\mathbf{1}} & x_A + x_D + x_E & -x_A & -x_E & -x_D  \cr
        \textcolor{blue}{\mathbf{2}} & -x_A & x_A + x_B & -x_B & 0  \cr
        \textcolor{blue}{\mathbf{3}} & -x_E & -x_B & x_B + x_C + x_E & -x_C  \cr
        \textcolor{blue}{\mathbf{4}} & -x_D & 0 & -x_C & x_C + x_D  \cr
    }
    \]

    \begin{align*}
    \mathcal{M}(G, 1) &= \left\{ 
    \bordermatrix{
         & \textcolor{blue}{\mathbf{2}} & \textcolor{blue}{\mathbf{3}} & \textcolor{blue}{\mathbf{4}} \cr
        \textcolor{blue}{\mathbf{2}} & x_A + x_B & -x_B & 0  \cr
        \textcolor{blue}{\mathbf{3}} & -x_B & x_B + x_C + x_E & -x_C  \cr
        \textcolor{blue}{\mathbf{4}} & 0 & -x_C & x_C + x_D  \cr
    } 
    : x_A, \ldots, x_E \in \CC \right\} \\
    &= \left\{ 
    \begin{pmatrix}
        y_1 & y_2 & 0 \\
        y_2 & y_3 & y_4 \\
        0 & y_4 & y_5
    \end{pmatrix} : y_1, \ldots, y_5 \in \CC \right\}.
    \end{align*}

    \begin{align*}
    \mathcal{M}(G, 2) &= \left\{
    \bordermatrix{
        & \textcolor{blue}{\mathbf{1}} & \textcolor{blue}{\mathbf{3}} & \textcolor{blue}{\mathbf{4}} \cr
        \textcolor{blue}{\mathbf{1}} & x_A + x_D + x_E & -x_E & -x_D  \cr
        \textcolor{blue}{\mathbf{3}} & -x_E & x_B + x_C + x_E & -x_C  \cr
        \textcolor{blue}{\mathbf{4}} & -x_D & -x_C & x_C + x_D  \cr
    } : x_A, \ldots, x_E \in \CC
    \right\}
    \\
    &= \left\{
    \begin{pmatrix}
        y_1 & y_2 & y_3 \\
        y_2 & y_4 & y_5 \\
        y_3 & y_5 & y_6
    \end{pmatrix} : \substack{y_1, \ldots, y_6 \in \CC \\ y_3 + y_5 + y_6 = 0}
    \right\}.
    \end{align*}
\end{example}

In \Cref{lemma: our models include grpahical models}, we show that the models considered in this section are generalizations of the celebrated family of Gaussian graphical models.

\begin{prop}
    \label{lemma: our models include grpahical models}
    Let $G$ be a simple weighted graph with vertex set $[n]$, and let $G^\prime$ be the graph obtained from $G$ by adding a new vertex, $0$, and edges connecting $0$ to every other vertex. Then $\mathcal{M}(G', 0)$ is the Gaussian graphical concentration model for $G$.
\end{prop}

\begin{proof}
For a given graph $G=([n],E)$, the Gaussian concentration model $\mathcal{L}_G$ is spanned by the matrices of the form $K_{ij}$ for $(i,j)\in E$, where each $K_{ij}$ is an $n \times n$ matrix with $1$ in the $(i,j)$ and $(j,i)$ entries and $0$ elsewhere, and $K_{ii}$, where $(K_{ii})_{jk} = 1$ if and only if $j = k = i$, and 0 otherwise. In other words, $\mathcal{L}_G$ is an $n+|E|$ dimensional subspace of the space of symmetric matrices where the vertex parameters $k_{ii}$ and edge parameters $k_{ij}$ are independent. 

Now, let $L(G')$ be the Laplacian of $G'$ and $L_{\widehat{0}, \widehat{0}}$ be the submatrix obtained by deleting the $0$th row and column of $L(G')$. We know from the properties of Laplacians that along with $L(G')_{i,j}$ being $0$ when $(i,j) \notin E$, the sum of each row of $L(G')$ is also zero. The second property of $L(G')$ is not satisfied by the Gaussian concentration models. However, as every vertex of $G' \setminus \{0\}$ is connected to $0$ by an edge, the $(0,i)$ entry of $L(G')$ is $-x_{(0,i)}$, which is non zero for every vertex $i$. This implies that the row sum of the $i$th row of $L_{\widehat{0}, \widehat{0}}$ is $x_{(0,i)}$, which is no longer zero. Thus, we consider the following morphism between $\mathcal{L}_G$ and $\mathcal{M}(G',0)$:
\begin{eqnarray*}
    \phi:\mathcal{L}_G &\rightarrow & \mathcal{M}(G',0) \\
    k_{ij} &\mapsto & \left\{
    \begin{array}{ll}
    -x_{ij}: i\neq j, (i,j)\in E \\
    \sum_{j\in \nei(i)\setminus \{0\}} x_{i,j} +x_{i,0}, \,  i=j. \\
    \end{array}
    \right.
\end{eqnarray*}
As $\phi$ is invertible, with $\phi^{-1}$ given by
\begin{eqnarray*}
    \phi^{-1}: \mathcal{M}(G',0) & \rightarrow & \mathcal{L}_{G} \\
    x_{ij} & \mapsto & \left\{
    \begin{array}{ll}
    -k_{ij} : j\in \nei(i)\setminus\{0\} \\
    k_{ii}-\sum_{j: (i,j)\in E} k_{ij} : j=0, \\
    \end{array}
    \right.
\end{eqnarray*}
we conclude that $\mathcal{M}(G',0)$ is isomorphic to the Gaussian concentration model for $G$.
\end{proof}

Kirchoff's matrix-tree theorem \cite{kirchoff} states that the determinant of any principal submatrix of $L(G)$ enumerates the spanning trees of $G$, and does not depend on the row and column removed.

\begin{definition}
    The \textit{spanning tree generating function} of a connected graph $G$, denoted by $P_G$, is the sum of indicator monomials for spanning trees of $G$.
    \begin{equation}
        \label{defn: spanning tree generating function P_G}
        P_G := \sum_{\substack{T \subset G \\ \text{spanning}}} \left( \prod_{e \in T} x_e \right).
    \end{equation}
\end{definition}

The fact that $P_G = \det L(G)_{\widehat{k},\widehat{k}}$ for any vertex $k$ of $G$, along with the proof of \Cref{prop: homaloidal to det} and \Cref{prop: homaloidal iff mld 1}, shows that if the model $\mathcal{M}(G,k)$ has ML degree one for some vertex $k$ of $G$, then $P_G$ is homaloidal. Moreover, $P_G$ being homaloidal implies that $\mathcal{M}(G,k)$ has ML degree one for every vertex $k$ of $G$. Subsection \ref{subsec: chordal then homaloidal} provides a sufficient condition on $G$ for $P_G$ being homaloidal.

\subsection{Criterion for ML Degree One} 
\label{subsec: chordal then homaloidal}

In this subsection, we prove \Cref{thm: chordal implies ML degree one}, which states that spanning tree models of chordal graphs have ML degree one. This is analogous to the situation with Gaussian graphical models, where models arising from chordal graphs also have ML degree one. For spanning tree models, this result was recently proved using matrix completion in \cite{hentschel2023statistical}. Our new proof uses homaloidal polynomials and relies heavily on \Cref{prop: homaloidal to det}. We begin this subsection by the definition of chordal graphs and Schur complements.

\begin{notation}
Let $G=(V,E)$ be a graph and $S\subseteq V$. The induced subgraph of $G$ on $S$ is denoted by $G[S]$ and the edge set of this subgraph is denoted by $E[S]$.
\end{notation}

\begin{definition}
    \label{defn: chordal graph}
    A graph $G$ is chordal (or decomposable) if every cycle in $G$ with at least four vertices has a chord. Equivalently, $G = (V,E)$ is chordal if it is a complete graph, or if there exists a partition $(I,J,K)$ of $V$ such that $G[I\cup J]$ and $G[J\cup K]$ are chordal, $G[J]$ is complete, and there are no edges between the vertices in $I$ and $K$.
\end{definition}

For examples of chordal and non-chordal graphs, see \Cref{fig: chordal examples}.

\begin{figure}
    \centering
    \includegraphics{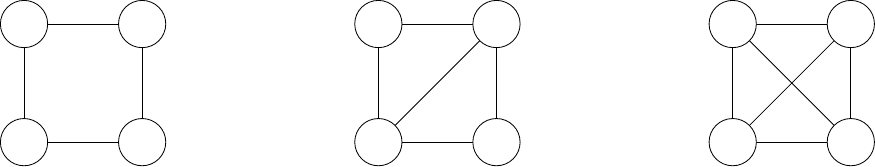}
    \caption{The 4-cycle (leftmost graph) is not chordal. The two graphs on the right are chordal.}
    \label{fig: chordal examples}
\end{figure}

\begin{definition}
    Let $M$ be an $m\times n$ matrix with the following block structure:
    \begin{align*}
        M = \begin{pmatrix}
            A & B \\ C & D
        \end{pmatrix},
    \end{align*}
    where $A$ and $D$ are square matrices. 
    \begin{enumerate}[i.]
        \item If $A$ is invertible, the Schur complement of $M$ with respect to $A$ is denoted by $M/A$, and defined as $D - CA^{-1}B$.
        \item If $D$ is invertible, the Schur complement of $M$ with respect to $D$ is denoted by $M/D$, and defined as $A - BD^{-1}C$.
    \end{enumerate}
\end{definition}

\begin{theorem}
    \label{thm: chordal implies ML degree one}
    If $G=([n],E)$ is a chordal graph, then $P_G$ is homaloidal; equivalently, the model $\mathcal{M}(G,k)$ has ML degree one for any $k\in [n]$.
\end{theorem}

For the proof of \Cref{thm: chordal implies ML degree one}, we require the following lemma.

\begin{lemma} \label{lemma: zero block inverse}
    Let 
    \begin{align*}
    M = \bordermatrix{
     & I & J & K \cr
     I & A & B & 0 \cr
     J & B^T & C & D \cr
     K & 0 & D^T & E 
    },
    \end{align*}
    where $A$, $C$ and $E$ are square matrices of arbitrary sizes. Define $N \coloneqq M^{-1}$. If $M$, $A$, and $E$ are invertible, then 
    \begin{align*}
        N_{I,K} = N_{I,J}\cdot N_{J,J}^{-1}\cdot N_{J,K}
    \end{align*}
\end{lemma}

\begin{proof}
Define
\begin{align*}
    & X \coloneqq A, \quad Y \coloneqq \begin{pmatrix} B & 0 \end{pmatrix}, \quad Z \coloneqq  \begin{pmatrix}C & D \\ D^T & E\end{pmatrix}.
\end{align*}
Then 
\begin{align*}
    M = \begin{pmatrix} X & Y \\ Y^T & Z \end{pmatrix}.
\end{align*}
Since $X$ and $M$ are both invertible, we have $M/X = Z - Y^T X^{-1} Y$ is also invertible. Let $(Z - Y^T X^{-1} Y)^{-1} = \begin{pmatrix} W_1 & W_2 \\ W_2^T & W_3\end{pmatrix}$, where $W_1$ and $W_3$ have the same sizes as $C$ and $E$ respectively. Therefore, 
\begin{align*}
    N = &\begin{pmatrix}
        X^{-1} + X^{-1} Y (Z - Y^T X^{-1} Y)^{-1} Y^T X^{-1} & - X^{-1} Y (Z - Y^T X^{-1} Y)^{-1} \\
        -(Z - Y^T X^{-1} Y)^{-1} Y^T X^{-1} & (Z - Y^T X^{-1} Y)^{-1}
    \end{pmatrix}
    \\
    = & \begin{pmatrix}
        A^{-1} + A^{-1} B W_1 B^T A^{-1} & -A^{-1}B W_1 & -A^{-1} B W_2 \\
        -W_1 B^T A^{-1} & W_1 & W_2 \\
        -W_2^T B^T A^{-1} & W_2 ^T & W_3
    \end{pmatrix},
\end{align*}
and we have $N_{I,K} = N_{I,J}\cdot N_{J,J}^{-1}\cdot N_{J,K}$ when $W_1$ is invertible. On the other hand, 
\[W_1 = \left(C- B^T A^{-1} B - D E^{-1} D^T\right)^{-1}\]
when $E$ is invertible. So, $W_1$ is invertible.
\end{proof}

\begin{proof}[Proof of \Cref{thm: chordal implies ML degree one}]
    For a chordal graph $G=([n],E)$ and any $k\in [n]$, 
    \begin{align*}
         \mathcal{M}(G,k)^{-1}= \overline{\left\{\Sigma \in {\sym{n-1}} \mid \Sigma^{-1} = \sum_{e\in E} x_e K_e^{G,k} \text{ for some } x_e\in \mathbb{C}.\right\}},
    \end{align*}
    where for $e=ij$, if $\{i,j\} \cap \{k\} = \emptyset$, then
    \begin{align*}
        (K_e^{G,k})_{rs} = \begin{cases}
            1 & \text{if $r=s=i$ or $r=s=j$} \\
            -1 & \text{if $\{r,s\} = \{i,j\}$} \\
            0 & \text{otherwise}
        \end{cases},
    \end{align*}
    and if $i=k$, then
    \begin{align*}
        (K_e^{G,k})_{rs} = \begin{cases}
            1 & \text{if $r=s=j$} \\
            0 & \text{otherwise}
        \end{cases}.
    \end{align*}
    We will prove $\mathcal{M}(G,k)$ has ML degree one for some $k\in [n]$. By \cite{améndola2023differential}, this will guarantee $P_G$ is a homaloidal polynomial. The proof is done by induction on the number of maximal cliques of $G$. 
    
    First, assume $G$ has 2 maximal cliques $C_1$ and $C_2$. Let $k\in C_1 \cap C_2$. The set of critical points of the log-likelihood function over $\mathcal{M}(G,k)$ is in bijection with the set of points $\Sigma = (\sigma_{ij} )_{i,j\in [n-1]} \in \mathcal{M}(G,k)^{-1}$ with $\langle \Sigma, K_e^{G,k}\rangle = \langle S, K_e^{G,k}\rangle$ for all $e\in E$ and a generic sample covariance $S = ( s_{ij} )_{i,j\in [n-1]}$ \cite[Page 2, (3)]{sturmfelsUhler}. For $e=ik\in E$, 
    \begin{align} \label{eq: trace1}
        \langle \Sigma, K_e^{G,k}\rangle = \langle S, K_e^{G,k}\rangle \Longleftrightarrow \sigma_{ii} = s_{ii}.
    \end{align}
    For $e=ij\in E$, where $ik,jk\in E$,
    \begin{align} \label{eq: trace2}
        \langle \Sigma, K_e^{G,k}\rangle = \langle S, K_e^{G,k}\rangle \Longleftrightarrow \sigma_{ii} + \sigma_{jj} - 2\sigma_{ij} = s_{ii} + s_{jj} - 2s_{ij}.
    \end{align}
    The matrix $\Sigma^{-1}$ has the following block structure:
    \begin{align*}
        \Sigma^{-1} = \bordermatrix{
     & C_1\setminus C_2 & C_1 \cap C_2 \setminus \{k\} & C_2 \setminus C_1 \cr
     C_1\setminus C_2 & \ast & \ast & 0 \cr
     C_1 \cap C_2 \setminus \{k\} & \ast & \ast & \ast \cr
     C_2 \setminus C_1 & 0 & \ast & \ast
    }
    \end{align*}
    By \eqref{eq: trace1} and \eqref{eq: trace2}, for all $i\in [n-1]$, the entries $\sigma_{ii}$, and for all $ij\in E$, the entries $\sigma_{ij}$, i.e. the blocks $\Sigma_{C_1\setminus \{k\},C_1\setminus \{k\}}$ and $\Sigma_{C_2\setminus \{k\},C_2\setminus \{k\}}$, are uniquely determined. By \Cref{lemma: zero block inverse}, this guarantees that $\Sigma_{C_1\setminus C_2,C_2\setminus C_1}$ and $\Sigma_{C_2\setminus C_1,C_1\setminus C_2}$ are uniquely determined too. So, the ML degree of $\mathcal{M}(G,k)$ is one in this case.

    Assume $P_G$ is homaloidal if $G$ is a chordal graph with at most $s$ maximal cliques. 
    
    Consider a chordal graph $G=([n],E)$ with $s+1$ maximal cliques. Since $G$ is chordal, there exists a partition $(I,J,K)$ of $[n]$ such that the graphs $G[I\cup J]$ and $G[J\cup K]$ are both chordal, $G[J]$ is a clique, and there are no edges between the vertices in $I$ and $K$.

    Let $j\in J$ and $S$ be a generic covariance matrix. The ML degree of $\mathcal{M}(G,j)$ is equal to $|\mathcal{CP}(G)|$, where 
    \begin{align*}
        \mathcal{CP}(G)\coloneqq \left\{ \Sigma \in \sym{n-1} \mid \langle K_e^{G,j}, \Sigma \rangle = \langle K_e^{G,j}, S \rangle \text{ for all $e\in E$, and } \Sigma^{-1}=L(G)_{\widehat{j}, \widehat{j}} (\bfx_E) \text{ for some } \bfx_E\in \mathbb{C}^E \right\} .
    \end{align*}
    Assume
    \begin{align*}
        L(G)_{\widehat{j}, \widehat{j}}(\bfx_E) = \bordermatrix{
     & I & J' & K \cr
     I & A & B & 0 \cr
     J' & B^T & C & D \cr
     K & 0 & D^T & E 
    }(\bfx_E),
    \end{align*}
    where $J'\coloneqq J\setminus \{j\}$. Let $K(G,j)(\bfx_E)\coloneqq \left(L(G)_{\widehat{j}, \widehat{j}}(\bfx_E)\right)^{-1}$. Using Schur complements, one can verify that 
    \begin{align*}
        &K(G,j)(\bfx_E)_{I \cup J', I \cup J'} = \begin{pmatrix} A & B \\ B^T & C-DE^{-1}D^T \end{pmatrix}^{-1}(\bfx_E), \text{ and } \\
        &K(G,j)(\bfx_E)_{J' \cup K, J' \cup K} = \begin{pmatrix} C-B^T A^{-1} B & D \\ D^T & E\end{pmatrix}^{-1}(\bfx_E).
    \end{align*}
    On the other hand, for all $\Sigma \in \mathcal{CP}(G)$, we have
    \begin{align*}
        \langle K_e^{G[I \cup J],j}, S_{I \cup J', I \cup J'}  \rangle = \langle K_e^{G[I \cup J],j}, \Sigma_{I \cup J', I \cup J'} \rangle \text{ for all $e\in E[I\cup J]$},
    \end{align*}
    and for some $\bfx_{E[I \cup J]}\in \mathbb{C}^{E[I \cup J]}$,
    \begin{align*}
        \left(\Sigma_{I\cup J', I \cup J'}\right)^{-1} = \left(K(G,j)(\bfx_E)_{I\cup J', I \cup J'}\right)^{-1} = \begin{pmatrix} A & B \\ B^T & C-DE^{-1}D^T \end{pmatrix}(\bfx_E).
    \end{align*}
    We will show that
    \begin{align}\label{eq: smaller model}
       \mathcal{M}(G[I\cup J],j)^{-1} &= \overline{\left\{ \Gamma \in S_{|I \cup J'|} \mid \Gamma^{-1} = \begin{pmatrix} A & B \\ B^T & C-DE^{-1}D^T \end{pmatrix}(\bfx_E) \text{ for some } \bfx_E\in \mathbb{C}^E \right\}}.
    \end{align}
    By induction hypothesis, $\mathcal{M}(G[I\cup J],j)$ has ML degree one. So, proving \eqref{eq: smaller model} will show that 
    $\Sigma_{I\cup J', I \cup J'}$ is uniquely determined. Similarly, $\Sigma_{J'\cup K, J'\cup K}$ is uniquely determined. Now the fact that $\Sigma$ is determined uniquely follows from \Cref{lemma: zero block inverse}.

    To show \eqref{eq: smaller model}, assume $E'$ denotes the set of edges in $G$ with one endpoint in $J'=J\setminus \{j\}$ and one endpoint in $K$. Let $\Gamma \in \mathcal{M}(G[I\cup J],j)^{-1}$. Then 
\begin{align*}
    \Gamma^{-1} = \begin{pmatrix}
        A & B \\ B^T & C
    \end{pmatrix} (\bfx_E)
\end{align*}
for some $\bfx_{E}\in \mathbb{C}^{E}$ with $\bfx_{E'} = 0$. So, considering that $D(\bfx_E) = 0$, one can write
\begin{align*}
    \Gamma^{-1} = \begin{pmatrix}
        A & B \\ B^T & C - DE^{-1}D^T
    \end{pmatrix}(\bfx_E).
\end{align*}
Therefore,
\begin{align*}
    \mathcal{M}(G[I\cup J],j)^{-1} \subseteq \overline{\left\{ \Gamma \in S_{|I \cup J'|} \mid \Gamma^{-1} = \begin{pmatrix} A & B \\ B^T & C-DE^{-1}D^T \end{pmatrix}(\bfx_E) \text{ for some } \bfx_E\in \mathbb{C}^E \right\}}.
\end{align*}
Now let 
\begin{align*}
    \Gamma^{-1} = \begin{pmatrix}
        A & B \\ B^T & C-DE^{-1}D^T
    \end{pmatrix}(\bfx_E)
\end{align*}
for some $\bfx_E\in \mathbb{C}^E$. We will show that $\Gamma\in \mathcal{M}(G[I\cup J],j)^{-1}$.  Note that one can write 
\begin{align*}
    \mathcal{M}(G[I\cup J],j)^{-1} =
        \overline{
        \left\{
        \Gamma \in S_{|I \cup J'|} \mid
        (\Gamma^{-1})_{uv} = 0 \text{ if } uv \not\in E, \ \ \sum_{v\in I\cup J'} (\Gamma^{-1})_{uv} = 0 \text{ if } uj\not\in E 
        \right\}
        }.
\end{align*}
Now since $G[J]$ is a complete graph and $j\in J$, we have
\begin{align*}
    \mathcal{M}(G[I\cup J],j)^{-1} 
        &=
        \overline{
        \left\{
        \Gamma \in S_{|I \cup J'|} \mid \text{ for all } u \in I,\ \ 
        (\Gamma^{-1})_{uv} = 0 \text{ if } uv \not\in E, \text{ and} \sum_{v\in I\cup J'} (\Gamma^{-1})_{uv} = 0 \text{ if } uj\not\in E
        \right\}
        }\\
        &= \overline{
        \left\{
        \Gamma \in S_{|I \cup J'|} \mid \left(\Gamma^{-1}\right)_{I, I \cup J'} = \left( L(G[I\cup J])_{{\widehat{j}, \widehat{j}}}\right)_{I, I \cup J'}(\bfx_E) \text{ for some }\bfx_E\in \mathbb{C}^E
        \right\}
        }\\
        &=\overline{
        \left\{
        \Gamma \in S_{|I \cup J'|} \mid \left(\Gamma^{-1}\right)_{I, I \cup J'} = \begin{pmatrix} A & B\end{pmatrix}(\bfx_E) \text{ for some }\bfx_E\in \mathbb{C}^E
        \right\}
        }.
\end{align*}
So, $\Gamma\in \mathcal{M}(G[I\cup J],j)^{-1}$, and we have
\begin{align*}
    \mathcal{M}(G[I\cup J],j)^{-1} \supseteq \overline{\left\{ \Gamma \in S_{|I \cup J'|} \mid \Gamma^{-1} = \begin{pmatrix} A & B \\ B^T & C-DE^{-1}D^T \end{pmatrix}(\bfx_E) \text{ for some } \bfx_E\in \mathbb{C}^E \right\}}.
\end{align*}
This concludes the proof of \eqref{eq: smaller model}, and hence, the proof of theorem.
\end{proof}

\begin{remark}
    Given a generic sample covariance matrix $S$, the proof of \Cref{thm: chordal implies ML degree one} suggests an inductive method for computing the MLE of $\mathcal{M}(G,j)$ for a chordal graph $G=([n],E)$ and $j \in [n]$ only when $j$ is adjacent to every other vertex in $G$. Note that in this case, $\mathcal{M}(G,j)$ is a Gaussian graphical concentration model by \Cref{lemma: our models include grpahical models}. Investigating possible relations between the MLE of $\mathcal{M}(G,j)$ and the MLE of $\mathcal{M}(G,j')$ for some $j' \in [n] \setminus \{j\}$ remains a topic for future work. 
\end{remark}

\subsection{ML Degree for Cycles}

Let $C_n$ be a cycle graph on $n$ vertices with $n \geq 3$. In this subsection, we explicitly compute the ML degree for $\mathcal{M}(C_n, k)$, which results in \Cref{thm: cycle ml degree}.

\begin{definition}
    \label{defn: Eulerian numbers}
    The $(n-1)$st Eulerian number is the number of subsets of $[n-1]$ with at least two elements, i.e., $2^{n-1}-n$. For $n = 3, 4, 5, 6, 7$ the Eulerian numbers are $1, 4, 11, 26,$ and $57$. This is OEIS sequence A000295 \cite{OEIS}.
\end{definition}

\begin{theorem}
    \label{thm: cycle ml degree}
    The ML degree of $\mathcal{M}(C_n, k)$, for any $k\in [n]$, is the $(n-1)$st Eulerian number.
\end{theorem}

If $f$ is a homogeneous polynomial of degree $d$ in $n + 1$ variables, then each coordinate of $\nabla f$ is homogeneous of degree $d - 1$. Therefore, the coordinates of $\nabla f$ define a rational map between projective spaces, which we denote by $\overline{\nabla f}: \PP^n \dashrightarrow \PP^n$. It is defined away from the simultaneous vanishing of the partial derivatives, $V(\partial f) := V(\frac{\partial f}{\partial x_i} \mid i = 0, \ldots, n)$.
In order to prove \Cref{thm: cycle ml degree}, we redefine the ML degree of $\mathcal{M}(C_n, k)$ as the degree of $\overline{\nabla P}_{C_n}: \PP^{n-1} \dashrightarrow \PP^{n-1}$.

\begin{lemma}
    \label{lem: ML deg is deg}
    Let $\mathcal{M} = \{ K(\bfx) = \sum_{i=0}^d x_i K_i \mid \bfx \in \mathbb{C}^{d+1} \}$ be a linear concentration model parameterized by linearly independent matrices $K_0, \ldots, K_d \in \mathrm{Sym}_n$, and let $f(\bfx) = \det K(\bfx)$ and $\bfu(S) = \nabla \mathrm{tr}(SK(\bfx))$. Assume that $\overline{\nabla f}: \PP^d \dashrightarrow \PP^d$ is dominant. The ML degree of $\mathcal{M}$ is equal to the degree of $\overline{\nabla f}: \PP^d \dashrightarrow \PP^d$.
\end{lemma}

\begin{proof}
    The ML degree of a Gaussian model is the number of critical points to the log-likelihood function, $\ell_{\mathcal{M}}(S)$, for generic $S \in \mathrm{Sym}_n$. The number of critical points of $\ell_{\mathcal{M}}(S)$ is the the number of solutions to the system of partial derivatives of the log-likelihood function:
    \begin{equation}
        \label{eqn: MLE equation}
        \frac{\partial}{\partial x_i} \left( \log f - \mbox{tr} S K \right) = 0.
    \end{equation}
    Rearranging the equation and writing it in vector form, we obtain the following equivalent system of equations:
    \begin{equation}
        \label{eqn: score equations}
        \frac{\nabla f(\bfx)}{f(\bfx)} = \bfu(S)
    \end{equation}
    (cf \cite[Theorem 4.2(a)]{améndola2023differential}. Note that since $K_0, \ldots, K_d$ are linearly independent, $\bfu(\mathrm{Sym}_n) = \CC^{d+1}$. Therefore, $\bfu(S)$ is generic for generic $S \in \mathrm{Sym}_n$ \cite[Theorem 2.3(a)]{améndola2023differential}. We claim that for generic $\bfu$, the above system has the same number of solutions as the following projective system:
    \begin{equation}
        \label{eqn: proj sys}
        \frac{\overline{\nabla f}(\bfx)}{f(\bfx)} = \bfu(S).
    \end{equation}
    Since $\overline{\nabla f}$ is dominant, $\mathrm{im}(\overline{\nabla f}) \setminus \overline{\nabla f} (V(f) \setminus V(\partial f))$ contains a dense open subset of $\PP^d$. So for generic $\bfu$, for any solution $\bfx$ to the system \eqref{eqn: proj sys} we have $f(\bfx) \neq 0$. 
    Thus, the system \eqref{eqn: proj sys} is equivalent to $\overline{\nabla f(\bfx)} = \bfu$ and the number of solutions is exactly equal to the degree of $\overline{\nabla f} : \PP^d \dashrightarrow \PP^d$.
    To prove the claim, we give a bijection between the solutions to \eqref{eqn: score equations} and the solutions to \eqref{eqn: proj sys}. A key observation is that $\overline{\nabla f}(\bfx)/f(\bfx)$ is homogeneous of degree $-1$. Suppose $\bfx, \bfy \in \CC^{d+1}$ are distinct and both satisfy \eqref{eqn: score equations}. Then $[\bfx], [\bfy] \in \PP^d$ satisfy \eqref{eqn: proj sys}. Furthermore, $[\bfx] \neq [\bfy]$; if $\bfx = \lambda \bfy$ for some $\lambda \in \CC^*$ then by homogeneity, 
    $$\frac{\nabla f (\bfx)}{f(\bfx)}  = \frac{1}{\lambda} \frac{\nabla f (\bfy)}{f(\bfy)} = \bfu = \frac{\nabla f (\bfy)}{f(\bfy)} \implies \lambda = 1 \implies \bfx = \bfy.$$
    On the other hand, if $[\bfx] \in \PP^d$ is a solution to \eqref{eqn: proj sys}, then let $\bfx \in \CC^{d+1}$ be any representative of $[\bfx]$. Since $[\bfx]$ solves \eqref{eqn: proj sys}, we have
    $$\frac{\nabla f(\bfx)}{f(\bfx)} = \lambda \bfu, ~~~ \exists \lambda \in \CC^\ast.$$
    Again using the homogeneity of $f$, $\bfy = \lambda \bfx$ is a solution to \eqref{eqn: score equations}, and it is the only solution to \eqref{eqn: score equations} that is a scalar multiple of $\bfx$. This proves the claim.
\end{proof}

In \Cref{thm: cycle dense} below, we prove that the gradient map is dominant for cycles. This, together with \Cref{lem: ML deg is deg}, proves that the ML degree of $\mathcal{M}(C_n, k)$ is equal to the degree of $\nabla P_{C_n}$ as a map of complex manifolds.

\begin{prop}
    \label{thm: cycle dense}
    The map $\overline{\nabla P}_{C_n} : \PP^{n-1} \dashrightarrow \PP^{n-1}$ is dominant.
\end{prop}

\begin{proof}
    We will show that the Hessian matrix of $P_{C_n}$ is invertible at a point, and then conclude via the inverse function theorem that the image of $\nabla P_{C_n}$ contains an open subset of $\CC^n$. The spanning trees of a cyclic graph on $n$ vertices are the size $(n-1)$ subsets of edges. Therefore,
    \begin{equation}
        \label{eqn: P_G cycle}
        P_{C_n} = \sum_{i = 1}^n x_1 \cdots \widehat{x}_i \cdots x_n.
    \end{equation}

    We can then compute the gradient, $\nabla P_{C_n}$, directly.

    \begin{equation}
        \label{eqn: nabla P_G coordinates cycle}
        \left( \nabla P_{C_n} \right)_i = \sum_{j: j \neq i} \left( \prod_{k: k \neq i, j} x_k \right).
    \end{equation}

    For any graph $G$, the diagonal entries of $H(P_{G})$ are always zero because $P_{G}$ is square-free (no spanning tree contains the same edge twice). The off-diagonal entries of $H(P_{C_n})$ can be computed directly from $\nabla P_{C_n}$.
    \begin{equation}
        \label{eqn: hessian cycle entries}
        H(P_{C_n})_{ik} = \frac{\partial \left( \nabla P_{C_n} \right)_i}{\partial x_k} =
        \begin{cases}
            \hphantom{..} 0 & \text{if } i = k, \\
            {\displaystyle \sum_{j : j \neq i, k}} \left( \displaystyle \prod_{\ell : \ell \neq i, j, k} x_\ell \right) & \text{otherwise.}
        \end{cases}
    \end{equation}

    Now evaluating $H(P_{C_n})$ at $\qq_n = (1,1, \ldots, 1)$, we find
    \[
    H(P_{C_n})_{ik} = 
        \begin{cases}
            \hphantom{..} 0 & \text{if } i = k, \\
            n - 2 & \text{otherwise.}
        \end{cases}
    \]

    This matrix has determinant $(-1)^{n-1} (n - 1) (n - 2)^n \neq 0$, so the inverse function theorem implies that $\nabla P_{C_n}$ is invertible near $(1 : \cdots : 1)$ as a map of complex manifolds. That means the image of $\overline{\nabla P}_{C_n}$ contains an open subset of $\PP^n$, which is Zariski dense. Thus, $\overline{\nabla P}_{C_n}$ is dominant.
\end{proof}

\Cref{lem: ML deg is deg} and \Cref{thm: cycle dense} show that in order to find the ML degree of the model $\mathcal{M}(C_n,k)$, it suffices to find the degree of the map $\nabla P_{C_n}$. \Cref{thm: cycle degree} finds the degree of this map.

\begin{theorem}
    \label{thm: cycle degree}
    The degree of $\overline{\nabla P}_{C_n}$ is the $(n-1)$st Eulerian number. Equivalently,  $\deg \overline{\nabla P}_{C_n}$ is the number of subsets of $[n-1]$ with at least two elements.
\end{theorem}

\begin{proof}
    We begin by noting that $\overline{\nabla P}_G : \PP^{n-1} \setminus V(\partial P_G) \rightarrow \PP^{n-1}$ is a holomorphic map of complex manifolds. Let $f: X \to Y$ be any map of complex manifolds. Recall that a \emph{regular value} of map $f$ is a point $y \in Y$ such that $f$ is non-singular at each point in the fiber of $y$, $f^{-1}(y) \subseteq X$. The degree of $f$ is the number of points in $f^{-1}(y)$, where $y \in Y$ is any regular value of $f$, counted with orientation $\pm 1$. Holomorphic maps preserve oriented angles, so every point in the fiber of a regular value of $\overline{\nabla P}_G$ is counted with positive orientation $+ 1$.

    We claim the all ones vector $\mathbbm{1}_n$ is a regular value of $\overline{\nabla P}_G$, and that its fiber is in bijection with subsets of $[n-1]$ with at least two elements. First, we compute the fiber $\left( \overline{\nabla P}_G \right)^{-1}(\qq_n)$, which is equivalent to solving the system of equations $(\nabla P_{C_n})_i = \lambda$ for $i = 1, \ldots, n$, $\exists \lambda \in \CC^\ast$. We will show that entries of $\bfx \in \left( \overline{\nabla P}_G \right)^{-1} \left( \qq_n \right)$ can take one of two values. Define the following sum for $i < n$, which will appear in each coordinate of $\nabla P_{C_n}$.
    \begin{equation*}
        \label{eqn: g_i defn}
        g_i := \sum_{j : j \neq i, n} \left( \prod_{k : k \neq i, j, n} x_k \right), \hspace{.5cm} i < n.
    \end{equation*}

    Writing $\nabla P_{C_n}$ using the $g_i$'s, we find

    \begin{equation}
        \label{eqn: grad P_G cycle equals ones}
        \left( \nabla P_{C_n} \right)_i = 
        \sum_{j : j \neq i} \left( \prod_{k : k \neq i, j} x_k \right) =
        \begin{cases}
            x_n g_i + x_1 \cdots \widehat{x}_i \cdots \widehat{x}_n & \text{for } i = 1, \ldots, n-1 \\
            x_j g_j + x_1 \cdots \widehat{x}_j \cdots \widehat{x}_n & \text{for } i = n, \text{ } \forall j = 1, \ldots, n-1.
        \end{cases}
    \end{equation}

    It follows from \eqref{eqn: grad P_G cycle equals ones} that for each $i < n$, either $x_i = x_n$ or $g_i = 0$.
    \begin{align*}
        \left( \nabla P_{C_n} \right)_n = \left( \nabla P_{C_n} \right)_i \iff & x_i g_i + x_1 \cdots \widehat{x}_i \cdots \widehat{x}_n = x_n g_i + x_1 \cdots \widehat{x}_i \cdots \widehat{x}_n \\
        \iff & x_i g_i = x_n g_i \\
        \iff & x_n = x_i \text{ or } g_i = 0.
    \end{align*}

    If $g_{i_1} = g_{i_2} = 0$, then 
    \begin{equation*}
        \left( \nabla P_{C_n} \right)_{i_j} = \lambda \iff x_1 \cdots \widehat{x}_{i_j} \cdots \widehat{x}_n = \lambda \iff \frac{x_{i_1}}{x_{i_2}} = 1 \iff x_{i_1} = x_{i_2}.
    \end{equation*}

    Thus, $[n]$ is partitioned into up to two sets on which $x_i$ is constant. The first part of the partition, $I$, consists of the coordinates equal to $x_n$, and the second part, $J$, consists of the coordinates not equal to $x_n$.
    \begin{equation*}
        \label{eqn: partition}
        I := \{ i : x_i = x_n \}, \hspace*{.5cm} J := \{ j : g_j = 0 \text{ and } x_j \neq x_n \} = \{ j : x_j \neq x_n \}.
    \end{equation*}

    By definition, $I$ is not empty. If $|I| = 1$ then $g_j = 0$ for all $j < n$. By adding up $n - 1$ different expressions for $\left( \nabla P_{C_n} \right)_n$ in \eqref{eqn: grad P_G cycle equals ones}, we obtain $(n - 1) \left( \nabla P_{C_n} \right)_n = \left( \nabla P_{C_n} \right)_n + \sum_{j < n} x_j g_j$. When $g_j = 0$ for all $j < n$, this implies $(n - 2)$$\left( \nabla P_{C_n} \right)_n = \sum_{j : j < n} x_j g_j = 0$, which is a contradiction. Therefore, $|I| \geq 2$. 
    
    On the other hand, if $J = \{ j \}$ then $\left( \nabla P_{C_n} \right)_j = (n-1)x_n^{n-2}$, but $\left( \nabla P_{C_n} \right)_i = (n-2)x_j x_n^{n-3} + x_n^{n-2}$. Setting these two equal yields $x_j = x_n$, a contradiction. If $J$ is empty, then $\bfx = \qq_n$, which is easily confirmed to be in $\left( \overline{\nabla P}_{C_n} \right)^{-1}(\qq_n)$.

    Now assume that $2 \leq |I|, |J|$. We use $g_j = 0$ to find the fixed ratio between the $I$ and $J$ coordinates. Let $\bfx_J \in \left( \overline{\nabla P}_{C_n} \right)^{-1}(\qq_n)$ be the point with $(\bfx_J)_i = a$ for all $i \in I$, and $(\bfx_J)_j = b$ for all $j \in J$; we claim that if $\bfx_J \in (\overline{\nabla P}_{C_n})^{-1}(\qq_n)$, then $a, b \neq 0$. 
    Suppose that $x_k = 0$ for some $k \in [n]$. Then by a formula similar to \eqref{eqn: grad P_G cycle equals ones}, $\left( \nabla P_{C_n} \right)_i = x_k g_i + x_1 \cdots \widehat{x}_i \cdots \widehat{x}_k \cdots x_n = x_1 \cdots \widehat{x}_i \cdots \widehat{x}_k \cdots x_n = \lambda$ for all $i \neq k$. But then $\left( \nabla P_{C_n} \right)_k = \sum_{i : i \neq k} x_1 \cdots \widehat{x}_i \cdots \widehat{x}_k \cdots x_n = (n-1) \lambda \neq \lambda$, which contradicts $\overline{\nabla P}_{C_n}(\bfx) = \qq_n$. Thus, $a, b \neq 0$. Now for any $j \in J$,
    \begin{align*}
        g_j = 0 \iff & \left( |I|  - 1 \right) a^{|I| - 2} b^{|J| - 1} + \left( |J| - 1 \right) a^{|I| - 1} b^{|J| - 2} = 0 \\
        \iff & \left( |I|  - 1 \right) b + \left( |J| - 1 \right) a = 0 \\
        \iff & \frac{a}{b} = \frac{|I| - 1}{1 - |J|} = \frac{|I| - 1}{1 - n + |I|}.
    \end{align*}

    For $n > 2$, the ratio $a/b$ is strictly decreasing as $|I|$ increases but not equal to 1 as long as $|I|, |J| \geq 2$.
    Reversing the steps above, we see that $\bfx \in (\overline{\nabla P}_{C_n})^{-1}(\qq_n)$ if and only if $\bfx = \bfx_J$ with $J = \emptyset$ (in which case $\bfx_J = \qq_n$) or $|J| \geq 2$.
    Therefore, the number of points in the fiber is the number of choices for $J \subset [n-1]$ with the only constraints being that $|J| \neq 1, n-1$. These subsets are in bijection with the subsets of $[n-1]$ of size at least two ($J = \emptyset \leftrightarrow [n-1]$).
    
    Now we prove that $\qq_n$ is a regular value of $\overline{\nabla P}_{C_n}$ by verifying that the Jacobian of $\nabla P_{C_n}: \CC^n \dashrightarrow \CC^n$ (which is the Hessian of $P_{C_n}$) is invertible at each point of the fiber. 
    For a point $\bfx_J = a \qq_I + b \qq_J$, the entries of the Jacobian of $\nabla P_{C_n}$ are
    \begin{align*}
        \label{eqn: jacobian entries}
        \frac{\left( \nabla P_{C_n} \right)_k}{\partial x_\ell}\vert_{\bfx = \bfx_J} = \frac{\partial^2 P_{C_n}}{\partial x_k \partial x_\ell}\vert_{\bfx = \bfx_J} =
        &\begin{cases}
            0 & \text{if } k = \ell, \\
            \left( |I| - 2 \right) a^{|I| - 3} b^{|J|} + |J| a^{|I| - 2} b^{|J| - 1} & \text{if } \left| \{ k, \ell \} \cap I \right| = 2, \\
            |I| a^{|I| - 1} b^{|J| - 2} + \left( |J| - 2 \right) a^{|I|} b^{|J| - 3} & \text{if } \left| \{ k, \ell \} \cap I \right| = 0, \\
            a^{|I| - 2} b^{|J| - 2} \left( \left( |I|  - 1 \right) b + \left( |J| - 1 \right) a \right) & \text{if } \left| \{ k, \ell \} \cap I \right| = 1
        \end{cases} \\
        =&\begin{cases}
            0 & \text{if } k = \ell, \\
            c & \text{if } \left| \{ k, \ell \} \cap I \right| = 2, \\
            d & \text{if } \left| \{ k, \ell \} \cap I \right| = 0, \\
            0 & \text{if } \left| \{ k, \ell \} \cap I \right| = 1.
        \end{cases}
    \end{align*}
    We claim that $c, d \neq 0$. First, if $J = \emptyset$ then only the first two cases above can occur. When $J$ is empty, $c = (|I| - 2) a^{|I| - 3} = n - 2 \neq 0$. Now assume that $|J| \geq 2$. Then we have
    $$c = 0 \iff a^{|I| - 3} b^{|J| - 1} ((|I| - 2) b + |J| a) = 0 \iff \frac{a}{b} = \frac{2 - |I|}{|J|} \iff \frac{|I| - 1}{1 - |J|} = \frac{2 - |I|}{|J|} \iff 2 = |I| + |J|.$$
    Since $|I|, |J| \geq 2$, the equivalences above prove that $c \neq 0$. Similarly, for $d$ we have
    $$d = 0 \iff a^{|I| - 1} b^{|J| - 3} (|I| b + (|J| - 2) a) = 0 \iff \frac{b}{a} = \frac{2 - |J|}{|I|} \iff \frac{1 - |J|}{|I| - 1} = \frac{2 - |J|}{|I|} \iff |I| + |J| = 2.$$
    Again, it is impossible to have $|I| + |J| = 2$ since $|I|, |J| \geq 2$, so it follows that $d \neq 0$.
    
    By rearranging the rows and columns, the Jacobian of $\nabla P_{C_n}$ at $\bfx_J$ is of the following form:
    \begin{center}
        \bordermatrix{
            & \textcolor{blue}{I} & \textcolor{blue}{J} \cr
            \textcolor{blue}{I} & \begin{array}{cccc}
                0 & c & \cdots & c \\
                c & 0 & & \vdots \\
                \vdots & & \ddots & c \\
                c & \cdots & c & 0
            \end{array} & \HugeZero \cr
            \textcolor{blue}{J} & \HugeZero & \begin{array}{cccc}
                0 & d & \cdots & d \\
                d & 0 & & \vdots \\
                \vdots & & \ddots & d \\
                d & \cdots & d & 0
            \end{array} \cr }
    \end{center}

    If $J$ is non-empty, the $II$ block and $JJ$ block have determinant $(-1)^{|I| - 1} (|I| - 1) c^{|I|}$ and $(-1)^{|J| - 1} (|J| - 1) d^{|J|}$ respectively. If $J$ is empty, then the determinant is $(-1)^{n - 1} (n - 1) c^{n-1} \neq 0$. It follows that the matrix has non-zero determinant, so $\qq_n$ is a regular value of $\overline{\nabla P}_{C_n}$. We conclude that the degree of $\overline{\nabla P}_{C_n}$ is the $(n-1)$st Eulerian number.
\end{proof}

\begin{proof}[Proof of \Cref{thm: cycle ml degree}]
    In \Cref{thm: cycle dense}, we proved that $\nabla P_{C_n}$ is dominant, so \Cref{lem: ML deg is deg} applies and the ML degree of $\mathcal{M}(C_n, k)$ is equal to the degree of $\nabla P_{C_n}$ as a map of complex manifolds. In \Cref{thm: cycle degree} we prove that the degree of $\nabla P_{C_n}$ is the $(n-1)$st Eulerian number, so we conclude that the ML degree of $\mathcal{M}(C_n, k)$ is the $(n-1)$st Eulerian number.
 \end{proof}

For $n \geq 4$, $C_n$ is not chordal, and our theorem confirms that the ML degree for the models corresponding to these graphs is greater than one. This fact, along with our computations, provides evidence for the conjecture that for a non-chordal graph $G$, $P_G$ is not homaloidal.

\begin{conjecture}
    \label{conj: ml deg one iff chordal}
    Let $G$ be an undirected graph. Then $P_G$, the spanning tree generating function corresponding to $G$, is homaloidal if and only if $G$ is chordal.
\end{conjecture}

The proof of \Cref{conj: ml deg one iff chordal} can be completed by showing that for any $S\subseteq V$ and any $k\in S$, the ML degree of $\mathcal{M}(G,k)$ is bounded below by the ML degree of $\mathcal{M}(G[S],k)$. In \cite{sturmfelsUhler}, it is shown that the ML degree of the Gaussian graphical model of $G[S]$ is a lower bound for the ML degree of the Gaussian graphical model of G. If proved, this conjecture, along with \Cref{lemma: our models include grpahical models}, would provide another proof for \cite[Theorem 4.3]{sturmfelsUhler}, where it is stated that a Gaussian graphical model corresponding to an undirected graph $G$ has ML degree one if and only if $G$ is chordal.

\subsection{Equations for the Covariance Model}

For a given undirected connected graph $G=(V,E)$, although the ML degree of $\mathcal{M}(G,k)$ is the same for any $k \in V$, the models $\mathcal{M}(G,k)$ might be different depending on the value of $k$. 
For instance, in the specific case where $k$ is connected to every other vertex in $G$, the concentration model $\mathcal{M}(G, k)$ is precisely the Gaussian concentration model for the graph $G[V \setminus \{ k \}]$ (\Cref{lemma: our models include grpahical models}). In the Gaussian graphical models, the (global Markov) conditional independence statements are obtained from the zeros of the concentration matrix. In \cite[Theorem 4.7]{pratikPreprint}, it is shown that the defining equations of the model satisfy the relation 
\begin{eqnarray}\label{eqnarray:graphical saturation}
I_G = CI_G : \langle \det(\Sigma)\rangle ^{\infty},
\end{eqnarray}
where $I_G$ and $CI_G$ are the vanishing ideal and conditional independence ideal (global Markov properties) of the graphical model, respectively. In the next theorem, we give a generalization of \eqref{eqnarray:graphical saturation} by showing a similar algebraic connection between the vanishing ideal $I(\mathcal{M}(G,k)^{-1})$ and the zeros appearing in the $k$th row and column of $L(G)$. In this subsection, when $G$ is clear from context, we denote $L(G)$ by $L = (l_{ij})_{1 \leq i, j \leq n}$.

\begin{theorem}
\label{thm: covariance ideal equations}
Let $G$ be a graph with $n$ vertices and $G'$ be the subgraph of $G$ obtained after removing the vertex $k$. Let $\Sigma$ be a symmetric matrix of variables indexed by the vertices of $G^\prime$, and
$\mathcal{R}_k$ be the set of polynomials in $\Sigma$ given by
\[
\mathcal{R}_k=\left\{\det \Sigma_{\widehat{i}, \widehat{j}} \mathrel{\big|} (i, j) \notin E, \text{ and } i,j\neq k \right\} \cup \left\{ \sum_i (-1)^{i+j} \det \Sigma_{\widehat{i}, \widehat{j}} \mathrel{\Big|} (j, k) \notin E \right\}.
\] 
Then the vanishing ideal of the model $I(\mathcal{M}(G,k)^{-1})$ is equal to $\langle \mathcal{R}_k \rangle : \langle \det(\Sigma)\rangle^{\infty}$.
\end{theorem}

\begin{proof}
Without loss of generality, let $k = n$. Define the following set of polynomials: 
\begin{equation*}
    \mathcal{P}_n := \left\{l_{ij} \mid (i,j) \notin E, \text{ and } i,j \neq n \right\} \cup \left\{\sum_i l_{ij} \mathrel{\Big|} (j,n) \notin E \right\}.
\end{equation*}

The first set above corresponds to the zeroes in $L_{\widehat{n}, \widehat{n}}$, and the second set corresponds to the deleted zeroes in the $n$th row and column.
We define a rational map $\rho: \mathbb{C}^{n-1 \times n-1}\dasharrow \mathbb{C}^{n-1 \times n-1}$ as
\[
\rho(L_{\widehat{n},\widehat{n}})=\rho(l_{11},l_{12},\ldots,l_{n-1n-1})=(\rho_{11}(L_{\widehat{n},\widehat{n}}),\rho_{12}(L_{\widehat{n},\widehat{n}}),\ldots, \rho_{n-1n-1}(L_{\widehat{n},\widehat{n}})),
\]
where $\rho_{ij}(L_{\widehat{n},\widehat{n}})$ denotes the $(i,j)$ coordinate of $\left(L_{\widehat{n},\widehat{n}}\right)^{-1}$. The pullback of $\rho$ can be written as 
\begin{eqnarray*}
\rho^*:\mathbb{C}[\sigma_{11},\sigma_{12},\ldots,\sigma_{n-1n-1}] &\rightarrow & \det(L_{\widehat{n},\widehat{n}})^{-1} \mathbb{C}[l_{11},l_{12},\ldots, l_{n-1n-1}] \\
\sigma_{ij} & \mapsto & \rho_{ij}(L_{\widehat{n},\widehat{n}}),
\end{eqnarray*}
where $\det(L_{\widehat{n},\widehat{n}})^{-1} \mathbb{C}[l_{11},l_{12},\ldots,l_{n-1n-1}]$ is the localization of the polynomial ring $\mathbb{C}[l_{11},l_{12},\ldots, l_{n-1n-1}]$ at the determinant of $L_{\widehat{n},\widehat{n}}$. 
(Note: We can change the indices of $\sigma_{ij}$ accordingly when $k \neq n$.)
Now, observe that the space $\mathcal{M}(G,n)$ is defined by setting the polynomials in $\mathcal{P}_n$ to zero, which forms a linear space in $\mathbb{C}^{n-1 \times n-1}$. Thus, the ideal $\langle \mathcal{P}_n \rangle$ is a prime ideal and is precisely the vanishing ideal of the set $\mathcal{M}(G,n)$. This allows us to define the restriction map $\rho|_{\mathcal{P}_n}:\mathcal{M}(G,n) \dasharrow \mathcal{M}(G,n)^{-1}$ which parameterizes our model. The pullback $\rho^*|_{\mathcal{P}_n}$ is
\begin{eqnarray*}
\rho^*|_{\mathcal{P}_n}: \mathbb{C}[\sigma_{11},\sigma_{12},\ldots, \sigma_{n-1n-1}] & \rightarrow & \det(L_{\widehat{n},\widehat{n}})^{-1} \mathbb{C}[l_{11},l_{12},\ldots, l_{n-1n-1}]/\langle \mathcal{P}_n \rangle \\
\sigma_{ij} & \mapsto &\rho_{ij}(L_{\widehat{n},\widehat{n}})+ \langle \mathcal{P}_n \rangle. 
\end{eqnarray*}
The kernel of $\rho^*|_{\mathcal{P}_n}$ is precisely the vanishing ideal of the model. Further, any polynomial $f$ lies in the kernel of $\rho^*|_{\mathcal{P}_n}$ if and only if $\rho^*(f)$ lies in $\langle \mathcal{P}_n \rangle$. Notice that as $\langle \mathcal{P}_n \rangle$ is generated by linear polynomials (and hence is a prime ideal), $\det(L_{\widehat{n},\widehat{n}})^{-1} \mathbb{C}[l_{11},l_{12},\ldots, l_{n-1n-1}]/\langle \mathcal{P}_n \rangle$ forms an integral domain and thus $\ker \rho^*|_{\mathcal{P}_n}$ is prime. 
The ideal $\langle \mathcal{R}_n \rangle$ lies in $\ker(\rho^*|_{\mathcal{P}_n})$ as 
the first set of generators of $\mathcal{R}_k$ gets mapped to $l_{ij}$ i.e., the first set of $\mathcal{P}_n$, while the second set gets mapped to $\sum_{i}{l_{ij}}$, the second set of $\mathcal{P}_n$.
In addition, as $\det(\Sigma)$ does not lie in $\ker(\rho^*|_{\mathcal{P}_n})$,
\[
\langle \mathcal{R}_n \rangle : \langle \det(\Sigma)^{\infty} \rangle \subseteq \ker(\rho^*|_{\mathcal{P}_n}): \langle \det(\Sigma)^{\infty} \rangle =  \ker(\rho^*|_{\mathcal{P}_n}). 
\]

To prove the reverse inclusion, we define the inverse of $\rho^*$ as
\begin{eqnarray*}
\left( \rho^\ast \right)^{-1}: \det(L_{\widehat{n},\widehat{n}})^{-1}\mathbb{C}[l_{11},l_{12},\ldots,l_{n-1n-1}] & \rightarrow & \det(\Sigma)^{-1} \mathbb{C}[\sigma_{11},\sigma_{12},\ldots,\sigma_{n-1n-1}] \\
l_{ij} & \mapsto & (i,j) \text{ coordinate of } \Sigma^{-1} \\
1/\det(L_{\widehat{n},\widehat{n}}) & \mapsto & \det(\Sigma).
\end{eqnarray*}
Now, for any $f\in \ker(\rho^*|_{\mathcal{P}_n})$, $\rho^*(f)$ lies in $\langle \mathcal{P}_n \rangle$. This means $\rho^*(f)$ is equal to $\sum_i f_i p_i$, where $p_i$ are the generators of $\langle \mathcal{P}_n \rangle$. Thus, we have
\[
f= \left( \rho^\ast \right)^{-1} \circ \rho^*(f) =  \left( \rho^\ast \right)^{-1} \left( \sum_i f_i p_i \right) = \sum_i \left( \rho^\ast \right)^{-1}(f_i)\left( \rho^\ast \right)^{-1}(p_i).
\]
If the polynomials $r_i$ are the generators of $\langle \mathcal{R}_n \rangle$, then observe that $\left( \rho^\ast \right)^{-1}(p_i)$ are precisely mapped to $r_i/\det(\Sigma)$. Thus, multiplying the equation above by enough powers of $\det(\Sigma)$ gives us that
\[
\det(\Sigma)^m f = \sum_i s_i r_i \in \langle \mathcal{R}_n \rangle,
\]
implying that $f$ lies in $\langle \mathcal{R}_n \rangle : \langle \det(\Sigma)^{\infty} \rangle$.
\end{proof}

In the next proposition, we display some polynomials which always lie in the vanishing ideal of the model, as they follow from the rank constraints of $L_{\widehat{k}, \widehat{k}}$ and $\Sigma$.
\begin{prop}\label{lemma:generators}
Let $G$ be a graph with $n$ vertices and $G'$ be the subgraph of $G$ obtained after removing the $k$th vertex. Let $\Sigma$ be a symmetric matrix of variables indexed by the vertices of $G^\prime$. 
\begin{enumerate}
\item The vanishing ideal (and the conditional independence ideal) of the Gaussian graphical model for $G'$ is contained in $I(\mathcal{M}(G,k)^{-1})$.
\item Let $\nei(k)$ denote the neighbors of $k$, i.e., the vertices that are adjacent to $k$. If $\#\nei(k)< n-1$, then \[
rank\begin{pmatrix}
    & &\Sigma_{\nei(k), \widehat{\emptyset}} & & \\
    1 & 1& \ldots & 1
\end{pmatrix} \leq \#\nei(k),
\]
i.e., the $(\#\nei(k) + 1) \times (\#\nei(k) + 1)$ minors of the above matrix vanish on $\mathcal{M}(G,k)^{-1}$.

\item Let the $(i,j)$ coordinate of $L$ be zero with $i,j \neq k$. If $i$ is not a neighbor of $k$, then
\[
rank \begin{pmatrix}
    & & \Sigma_{\widehat{i},\widehat{j}} & & \\
    1 & 1 & \ldots & 1
\end{pmatrix} \leq n-3,
\]
i.e., the $ (n-2) \times (n-2)$ minors of the above matrix vanish on $\mathcal{M}(G,k)^{-1}$.
\end{enumerate}
\end{prop}
\begin{proof}
\begin{enumerate}
    \item This directly follows from the fact that $\mathcal{M}(G,k)$ is contained in the Gaussian graphical model corresponding to $G'$, as the non-edges of $G'$ are also non-edges of $G$.

    \item We know that the $(i,k)$ coordinate of $L$ is zero (as a polynomial) if and only if the edge $(i,k)$ is missing in the graph $G$. As the sum of each row of $L$ is zero, the sum of the row of $L_{\widehat{k}, \widehat{k}}$ corresponding to the $i$th vertex is also zero (since we only remove a zero from that row in $L$). This immediately gives us a polynomial in $\mathbb{C}[\Sigma]$, which vanishes on $\mathcal{M}(G,k)^{-1}$. This polynomial is
    \begin{align*}
        \sum_{j} \mbox{adj}(\Sigma)_{i,j} = 0,
    \end{align*}
    where $\mbox{adj}(\Sigma)$ is the adjugate matrix of $\Sigma$.
    
    We can represent this equation as a determinantal constraint obtained by replacing the $i$th row of $\Sigma$ with the vector $(1,1,\ldots,1)$. Denoting this matrix as $\Sigma_{i\rightarrow 1}$, we know that $\Sigma_{i\rightarrow 1}$ has rank $n-2$. However, $\Sigma$ is an invertible matrix, implying that $\Sigma_{\widehat{i},\widehat{\emptyset}}$   
has rank $n-2$. This gives us that the vector $(1,1,\ldots,1)$ is a linear combination of all other rows of $\Sigma_{i\rightarrow 1}$.

    Now, the above statement is true for every missing edge $(i,k)$, i.e., the vector $(1,1,\ldots, 1)$ is a linear combination of all other rows of $\Sigma_{i\rightarrow 1}$, for every missing edge $(i,k)$. As the rows corresponding to neighbors of $k$ in $G$ are the only rows which are common in every $\Sigma_{i\rightarrow 1}$, we can deduce that $(1,1,\ldots, 1)$ is a linear combination of the rows corresponding to neighbors of $k$ in $G$, which completes the proof.

    \item As the $(i,j)$ coordinate of $L$ is zero, we know that $(-1)^{i+j}\mbox{adj}(\Sigma)_{i,j}= \det \Sigma_{\widehat{i},\widehat{j}}$ is zero. So, the rank of $\Sigma_{\widehat{i},\widehat{j}}$ is at most $n-3$. Now, we know from \ref{lemma:generators} $(2)$ that the rank of $\Sigma_{\nei(k),\widehat{\emptyset}}$ with the added row $(1,1,\ldots, 1)$ is at most $\#\nei(k)$, and that the vector $(1,1,\ldots, 1)$ is a linear combination of the rows of $\Sigma_{ne(k),\hat{\emptyset}}$. Thus, deleting the column corresponding to the vertex $j$ gives us that the rank of the matrix $\Sigma_{\nei(k),\widehat{j}}$ along with the row $(1,1,\ldots, 1)$ is still at most $\#\nei(k)$. As $\nei(k)\subseteq [n]\setminus \{i,k\}$, the rows of $\Sigma_{\nei(k),\widehat{j}}$ are also present in $\Sigma_{\widehat{i},\widehat{j}}$. This implies that adding the vector $(1,1,\ldots, 1)$ as a new row in $\Sigma_{\widehat{i},\widehat{j}}$ does not increase the rank of the matrix, and hence we have the desired minors in $I(\mathcal{M}(G,k)^{-1})$.  
\end{enumerate}
\end{proof}

\begin{example}
Let $G$ be the graph in \Cref{example:chordal laplacian}. As the vertex $1$ is connected to every other vertex in $G$, the vanishing ideal of $\mathcal{M}(G,1)^{-1}$ is exactly equal to the Gaussian graphical vanishing ideal of the subgraph of $G$ obtained after removing $1$. Thus, we have
\[
I(\mathcal{M}(G,1)^{-1})=\langle \sigma_{24}\sigma_{33}-\sigma_{23}\sigma_{34} \rangle.
\]
Similarly, $\mathcal{M}(G,2)$ is obtained by removing the row and column corresponding to the vertex $2$ from the Laplacian $L$. Observe that as $l_{24}$ is zero, removing the second row and column implies that the sum of the fourth row is still zero in $L_{\widehat{2},\widehat{2}}$. This gives us that $\mbox{adj} (\Sigma)_{1,4}+ \mbox{adj} (\Sigma)_{3,4} + \mbox{adj} (\Sigma)_{4,4}$ lies in $I(\mathcal{M}(G,2)^{-1})$. Now, as $\#\nei(2) = 2 < 3$, we can use \Cref{lemma:generators} (2) to show that $3 \times 3$ minors of the matrix
\[
\begin{pmatrix}
    \sigma_{11} & \sigma_{13} & \sigma_{14} \\
    \sigma_{13} & \sigma_{33} & \sigma_{34} \\
    1 & 1 & 1
\end{pmatrix}
\]
also lie in $I(\mathcal{M}(G,2)^{-1})$. Upon computing the ideal, we get that this determinant ($3\times 3$ minor) is precisely the generator of $I(\mathcal{M}(G,2)^{-1})$, i.e.,
\[
I(\mathcal{M}(G,2)^{-1})= \langle \sigma_{13}^2-\sigma_{13}\sigma_{14}-\sigma_{11}\sigma_{33}+\sigma_{14}\sigma_{33}+\sigma_{11}\sigma_{34}-\sigma_{13}\sigma_{34} \rangle.
\]
\end{example}

In the case of Gaussian graphical models, the vanishing ideal and conditional independence ideal are not necessarily equal. Further, it is still an open problem to get a graphical characterization of all the generators of the vanishing ideal. As the graphical vanishing ideal of $G[V \setminus \{ k \}]$ is always contained in $I(\mathcal{M}(G,k)^{-1})$, \Cref{lemma:generators} $(2)$ and $(3)$ do not provide a complete description of all the possible generators of $I(\mathcal{M}(G,k)^{-1})$. Thus, we leave it as an open problem to obtain a graphical characterization of all the generators of the ideal $I(\mathcal{M}(G,k)^{-1})$.

\section{Examples of Homaloidal Families} 
\label{sec: examples of homaloidals}

A hypersurface is homaloidal if its defining polynomial is homaloidal. In this section, we provide several examples of homaloidal hypersurfaces that can be realized as (affine) linear concentration models using our results in \Cref{sec: SDR for homaloidals}. First, we restate Dolgachev's classification \cite{cremonaTransformations} of homaloidal plane curves, and then we explicitly construct a four dimensional homaloidal hypersurface using techniques from \cite{ciliberto2007homaloidalHA}. Finally, we use \Cref{prop: homaloidal to det} and \Cref{prop: homaloidal iff mld 1} to prove that the product of two homaloidal polynomials in separate variables is again homaloidal.

\subsection{Dolgachev's Classification}

In \cite{cremonaTransformations}, Dolgachev classifies homaloidal plane curves (See \Cref{thm: homaloidal-classification} for $r=2$.). These plane curves can be naturally extended to hypersurfaces in $\PP^r$ for $r \geq 2$ \cite[Remark 3.2]{ciliberto2007homaloidalHA}, although other homaloidal hypersurfaces exist for $r > 2$.

\begin{theorem}[\cite{cremonaTransformations}]
    \label{thm: homaloidal-classification}
    The following hypersurfaces $V(f) \subset \PP^r$ are homaloidal for any $r \geq 2$:
    \begin{enumerate}[i.]
        \item a smooth quadric;
        \item the union of a smooth quadric with one of its tangent planes;
        \item the union of $r + 1$ independent hyperplanes.
    \end{enumerate}

    When $r = 2$, these are the only homaloidal hypersurfaces.
\end{theorem}

\begin{remark}
    A quadric polynomial $q(\bfx)$ can be written in matrix form $q(\bfx) = \bfx^\top Q \bfx$ for some symmetric matrix $Q$. Therefore, $\nabla q(\bfx) = 2 Q \bfx$, which proves that $q$ is homaloidal if and only if $Q$ is invertible. The \textit{rank of $q$} is the rank of $Q$.
\end{remark}

\begin{example}[Homaloidal plane curves]
    \label{ex: homaloidal plane curves}
    $f(x,y,z) = \frac{7}{25} x^2 - y^2 - \frac{48}{25} xz - \frac{7}{25} z^2$ is a smooth quadric (we see in \Cref{ex: quadric sdr} that $f$ is full rank), so $f$ is homaloidal. The plane $2x - 14z = 0$ is tangent to $V(f)$ at $(7 : 0 : 1)$, so $g(x,y,z) = f(x,y,z)(2x - 14z)$ is also homaloidal. $h(x,y,z) = (x + y) (y + z) (x + 2z)$ is a union of 3 independent hyperplanes, so it is homaloidal. 
\end{example}

\subsection{Constructing New Homaloidal Polynomials}

In \cite{ciliberto2007homaloidalHA}, the authors construct new homaloidal polynomials using projections of a \textit{rational normal scroll surface}, $S(a,b)$ for $a < b$. The resulting surface is denoted $Y(a,b)$, and its dual, $Y(a,b)^*$, is homaloidal for certain choices of $a, b$ (see below, \cite[Theorem 3.13]{ciliberto2007homaloidalHA}). We use this construction and \texttt{Macaulay2} to compute a specific example of a homaloidal polynomial which does not fall under \Cref{thm: homaloidal-classification} (the code is located at: \url{https://github.com/shelbycox/Homaloidal/blob/main/src/Sab.m2}).

\begin{theorem}{\cite[Theorem 3.13]{ciliberto2007homaloidalHA}} 
    \label{thm: homaloidal 1}
    $Y(a,b)$ is a surface obtained via two specific consecutive projections from a rational normal scroll surface $S(a,b)$. The hypersurface $Y(r - 2, d - r + 2)^\ast$ of degree $d$ is homaloidal for every $r \geq 3$, $d \geq 2r - 3$.
\end{theorem}

In the following subsection, we explain how such a homaloidal hypersurface can be computed using the construction in \cite[Sections 1, 3]{ciliberto2007homaloidalHA}. We begin by describing the scroll surface $S(a,b)$, and after that we introduce the tools we need to obtain $Y(a,b)$. We include a running example, for $a = 2$, $b = 3$, with explicit computations for each stage of the construction.

The ideal of $S(a,b)$ is generated by the $(2 \times 2)$-minors of the matrix below: 
\[ M(a,b) =
    \begin{pmatrix}
        x_0 & x_1 & \ldots & x_{a - 1} & x_{a + 1} & x_{a + 2} & \ldots & x_{a + b} \\
        x_1 & x_2 & \ldots & x_a & x_{a + 2} & x_{a + 3} & \ldots & x_{a + b + 1}
    \end{pmatrix}.
\]
$S(a,b)$ has degree $a + b$ and belongs to a special class of toric varieties called \textit{Hirzebruch surfaces}. It is ruled by lines between two curves.

\begin{definition}
    \label{defn: rational normal curve}
    The rational normal curve of degree $a$, denoted $C_a$, is the closure of the image of the map $(s : t) \mapsto (s^a : s^{a-1} t : \cdots : t^a)$. Its ideal is generated by the $2$-minors of the left side of $M(a,b)$.
\end{definition}

The surface $S(a,b)$ consists of lines between $C_a$ (embedded in the first $a + 1$ coordinates), and $C_b$ (embedded in the last $b + 1$ coordinates). In particular, $S(a,b)$ contains the points $V(x_0 = x_1 = \cdots = x_a, x_{a+1} = \cdots = x_{a + b + 1} = 0)$ and $V(x_0 = x_1 = \cdots = x_a = 0, x_{a+1} = \cdots = x_{a + b + 1})$. Furthermore, the line between these two points is a ruling on $S(a,b)$.

Using the $2$-minors of $M(a,b)$, we can also find points that are not on $S(a,b)$. For example, the points $V(x_0, \ldots, \widehat{x_i}, \ldots, x_{a + b + 1})$ are not on $S(a,b)$ for $i = 1, \ldots, a-1, a + 2, \ldots, a + b$, since the minor with $x_i$ on the anti-diagonal does not vanish.

\begin{example}
    Let $a = 2$ and $b = 3$. In step 1 of our code, we compute the defining ideal of $S(2,3)$. The matrix $M_{2,3}$ is as follows:
    \[ M_{2,3} =
        \begin{pmatrix}
            x_0 & x_1 & x_3 & x_4 & x_5 \\
            x_1 & x_2 & x_4 & x_5 & x_6
        \end{pmatrix}.
    \]
    The ideal of the surface $S(2,3)$ is generated by the $2 \times 2$ minors of $M_{2,3}$, which are listed below:
    \begin{eqnarray*}
    I(S(2,3)) &=& \langle x_{0}x_{2} - x_{1}^{2}, \:x_{0}x_{4} - x_{1}x_{3}, \:x_{1}x_{4} - x_{2}x_{3}, \:x_{0}x_{5} - x_{1}x_{4}, \:x_{1}x_{5} - x_{2}x_{4}, \:x_{3}x_{5} - x_{4}^{2},\\
    && \:x_{0}x_{6} - x_{1}x_{5}, \:x_{1}x_{6} - x_{2}x_{5}, \:x_{3}x_{6} - x_{4}x_{5}, \:x_{4}x_{6} - x_{5}^{2}\rangle.
    \end{eqnarray*}

    We can find points on $S(2,3)$ by finding points on one of the two rational normal curves in $S(2,3)$, which can be found using the parameterization of these rational normal curves. Below are two such points, one on the rational normal curve of degree 2, which lies in the first three coordinates, and one on the rational normal curve of degree 3, which lies in the last four coordinates. To find coordinate points not on $S(2,3)$, we use the matrix description. For each $2 \times 2$ minor with a constant anti-diagonal, setting the variable on the anti-diagonal equal to one with all other coordinates equal to zero gives a point not on $S(2,3)$.
    \begin{align*}
        (1 : 1 : 1 : 0 : 0 : 0 : 0), (0 : 0 : 0 : 1 : 1 : 1 : 1) &\in S(2,3), \\
        (0 : 1 : 0 : 0 : 0 : 0 : 0), (0 : 0 : 0 : 0 : 1 : 0 : 0), (0 : 0 : 0 : 0 : 0 : 1 : 0) &\notin S(2,3).
    \end{align*}
\end{example}

Using $S(a,b)$, one can construct another surface $Y(a,b)$ whose dual is homaloidal.

\begin{definition}
    $Y(a,b)$ is any surface obtained by the consecutive projections of $S(a,b)$ outlined in \eqref{eqn: homaloidal projections}.
    \begin{equation}
        \label{eqn: homaloidal projections}
        S(a, b) \xrightarrow{\sigma_\Phi} X(a, b) \xrightarrow{\sigma_\Psi} Y(a, b).
    \end{equation}
    The linear spaces $\Phi$ and $\Psi$ are discussed further in \Cref{susbec: first projection,subsec: second projection}. We note here that there are many possible choices for $\Phi$ and $\Psi$ for each $a, b$.
\end{definition}

The following subsections define these projections and the tools we need to compute them.

\subsubsection{Duals}

\begin{definition}
    \label{defn: conormal variety}
    The conormal variety $N(X)$ is the closure of the points of the form $(x, \pi) \in \PP^r \times \left( \PP^r \right)^\ast$, where $x$ is a smooth point of $X$ and $\pi^\perp$ contains the tangent space of $X$ at $x$. In more compact notation,
    \begin{equation}
        \label{eqn: conormal variety}
        N(X) := \overline{\{ (x, \pi) \mid x \in \mbox{smooth}(X), T_{X, x} \subset \pi^\perp \}} \subseteq \PP^r \times \left( \PP^r \right)^\ast.
    \end{equation}

    Recall that a point $x$ of $X$ is \textit{smooth} if $\dim T_{X,x} = \dim X$.
\end{definition}

\begin{definition}
    \label{defn: dual variety}
    The dual variety $X^\ast$ of $X$ is the projection of $N(X)$ onto its second factor $\left( \PP^r \right)^\ast$.
\end{definition}

In \texttt{Macaulay2}, duals can be computed using the \texttt{dualVariety} command in the \texttt{Resultants} package \cite{ResultantsSource, ResultantsArticle}.

\begin{example}[$a = 2, b = 3$]
    The dual $S(a,b)^\ast$ is a hypersurface (as expected). Step 2a in our code computes $S(2,3)^\ast$, which is defined by the following equation.
    \begin{eqnarray*}
    I(S(2,3)^\ast) &=& \langle x_{2}^{3}x_{3}^{2}-x_{1}x_{2}^{2}x_{3}x_{4}+x_{0}x_{2}^{2}x_{4}^{2}+x_{1}^{2}x_{2}x_{3}x_{5}-2\,x_{0}x_{2}^{2}x_{3}x_{5}-x_{0}x_{1}x_{2}x_{4}x_{5}+x_{0}^{2}x_{2}x_{5}^{2}\\
    &&-x_{1}^{3}x_{3}x_{6}+3\,x_{0}x_{1}x_{2}x_{3}x_{6}+x_{0}x_{1}^{2}x_{4}x_{6}-2\,x_{0}^{2}x_{2}x_{4}x_{6}-x_{0}^{2}x_{1}x_{5}x_{6}+x_{0}^{3}x_{6}^{2} \rangle.
      \end{eqnarray*}
\end{example}

\subsubsection{Projections}

Here we define the \textit{projection away from a linear space}, which will be used twice in the homaloidal hypersurface construction of \cite{ciliberto2007homaloidalHA}.

\begin{definition}
    \label{defn: projection}
    For a linear space $\Pi$ of dimension $m$ cut out by homogeneous linear forms $\ell_1, \ldots, \ell_{r - m}$, the \textit{projection away from $\Pi$} is the map $\sigma_{\Pi}$ below.
    \begin{center}
        \begin{tikzcd}[row sep = .25em]
            \sigma_\Pi :& \mathbb{P}^r \rar[dashed] & \left( \Pi^\perp \right)^\ast \cong \PP^{r - m - 1} \\
            & \bfx \rar[mapsto] & \left( \ell_1(\bfx) : \cdots : \ell_{r - m}(\bfx) \right)
        \end{tikzcd}
    \end{center}

    When $\Pi \cap X = \emptyset$, the closure of $\sigma_\Pi(X)$ is called the \textit{external projection of $X$ from $\Pi$}, and is denoted $X_\Pi$.
\end{definition}

In \texttt{Macaulay2}, projection away from the coordinate linear space $\Pi = \langle x_0, \ldots, x_{k} \rangle$ can be computed by eliminating the variables $x_{k+1}, \ldots, x_r$. If $\Pi$ is not a coordinate subspace, first use a linear transformation that takes $\Pi$ to a coordinate linear space, then use elimination. Alternatively, we can use \Cref{prop: crs prop1-1} to compute these projections.

\begin{prop}[{\cite[Proposition 1.1]{ciliberto2007homaloidalHA}}]
    \label{prop: crs prop1-1}
    Let $X \subset \PP^r$, and assume that the linear span of $X$ is $\PP^r$. If $\dim(X) < r - \dim \Pi - 1$, and $\Pi^\perp \cap X^\ast$ is irreducible and reduced, then $(X_\Pi)^\ast = \Pi^\perp \cap X^\ast$.
\end{prop}

\subsubsection{First Projection}
\label{susbec: first projection}
The first projection uses a remarkable property of $S(a,b)^\ast$: the singularities in $S(a,b)^\ast$ of multiplicity at least $b$ form a linear space, $E^\perp$ \cite[Proposition 1.6]{ciliberto2007homaloidalHA}. After computing $E$ (the orthogonal complement of $E^\perp$), there is a linear space $\Phi$ with the following properties:
\begin{enumerate}[i.]
    \item $\dim \Phi = a - 2$,
    \item $\Phi \subseteq E$, and
    \item $\Phi \cap S(a,b) = \emptyset$.
\end{enumerate}

The resulting surface, $X(a,b) := \overline{\sigma_\Phi(S(a,b))}$, also has degree $d$ and is also ruled. In particular, the image of rulings on $S(a,b)$ under the projection are lines on $X(a,b)$. The linear space $E$ maps to a line $\Lambda$ under the projection. The details of this construction can be found in \cite[Section 1]{ciliberto2007homaloidalHA}.

\begin{example}[$a = 2, b = 3$]
    In step 2b of our code, we compute $E = \langle x_3, x_4, x_5, x_6 \rangle$; since $a - 2 = 0$, $\Phi$ is a point. One possible choice is 
    $$\Phi = (0 : 1 : 0 : 0 : 0 : 0 : 0) = V(x_0, x_2, x_3, x_4, x_5, x_6).$$
    This point clearly lies on $E$, and we saw earlier that it is not a point on $S(2,3)$ (this is confirmed in step 3 of our code). We note that the choice of a coordinate point simplifies the computation that follows. In step 4 of our code, we eliminate $x_1$ from $I(S(2,3))$ to compute $I(X(2,3))$.
    $$I(X(2,3)) = \langle x_{0}x_{5} - x_{2}x_{3}, \,x_{2}x_{4} - x_{0}x_{6}, \,x_{4}^{2} - x_{3}x_{5}, \,x_{3}x_{6} - x_{4}x_{5}, \,x_{5}^{2} - x_{4}x_{6} \rangle.$$
\end{example}

\subsubsection{Second Projection}
\label{subsec: second projection}
The second projection uses a linear space $\Psi$ constructed from $\Lambda$ and $b - a$ rulings on $X(a,b)$. Denote the rulings on $X(a,b)$ by $F_1 \ldots, F_{b-a}$. Then there exists a generic linear space $\Psi$ satisfying the following conditions:
\begin{enumerate}[i.]
    \item $\dim \Psi = b - a - 1$,
    \item $\Psi \subseteq \mbox{span} \{ \Lambda, F_1, \ldots, F_{b-a} \}$, and
    \item $\Psi \cap X(a,b) = \emptyset$.
\end{enumerate}

The resulting surface is $Y(a,b) := \overline{\sigma_{\Psi}(X(a,b))}$, and its dual, $Y(a, b)^\ast$, is homaloidal when $a = r - 2$, $b = d - r + 2$ for any $r \geq 3$, $d \geq 2r - 3$ \cite[Theorem 3.13]{ciliberto2007homaloidalHA}. For more details on this stage of the construction, see \cite[Section 3]{ciliberto2007homaloidalHA}.

\begin{example}[$a = 2, b = 3$]
    In step 5, we compute the projection of $E$, $\Lambda = \langle x_3, x_4, x_5, x_6 \rangle$.
    Note that $b - a = 1$, so $\Psi$ is a point and we need just one ruling on $X(2,3)$, which we compute in step 6. Let $L$ denote the line on $S(a,b)$ containing the points $(1 : 1 : 1 : 0 : 0 : 0 : 0)$, and $(0 : 0 : 0 : 1 : 1 : 1 : 1)$. Then $F = \overline{\sigma_\Phi(L)}$ is a ruling on $X(2,3)$. In step 7, we choose the point $\Psi = (1 : -1 : 1 : 1 : 1 : 1)$ in the linear space spanned by $F$ and $\Lambda$. In step 8 of our code, we use a linear transformation to take $\Psi$ to the coordinate point $(0 : 1 : 0 : 0 : 0 : 0)$, and then eliminate $x_2$ to obtain $Y(2,3)$. In step 9, we take the dual to produce $Y(2,3)^\ast$.
    \begin{align*}
        I(Y(2,3)^\ast) =&\langle x_{0}^{3}x_{3}^{2}+x_{0}^{3}x_{4}^{2}+x_{0}^{2}x_{3}x_{4}^{2}+x_{0}^{2}x_{4}^{3}-2\,x_{0}^{3}x_{3}x_{5}-2\,x_{0}^{2}x_{3}^{2}x_{5}-2\,x_{0}^{2}x_{3}x_{4}x_{5}+x_{0}^{2}x_{4}^{2}x_{5} +x_{0}^{3}x_{5}^{2}
        +x_{0}x_{3}^{2}x_{5}^{2} \\
        &+2\,x_{0}^{2}x_{4}x_{5}^{2}+2\,x_{0}x_{3}x_{4}x_{5}^{2}+x_{0}x_{4}^{2}x_{5}^{2}+2\,x_{0}^{2}x_{5}^{3}+2\,x_{0}x_{3}x_{5}^{3}+2\,x_{0}x_{4}x_{5}^{3}+x_{0}x_{5}^{4}-2\,x_{0}^{3}x_{4}x_{6} \\
        &-4\,x_{0}^{2}x_{3}x_{4}x_{6}-2\,x_{0}x_{3}^{2}x_{4}x_{6}-3\,x_{0}^{2}x_{4}^{2}x_{6}-4\,x_{0}x_{3}x_{4}^{2}x_{6}-2\,x_{0}x_{4}^{3}x_{6}-2\,x_{0}^{2}x_{3}x_{5}x_{6}-4\,x_{0}^{2}x_{4}x_{5}x_{6}\\
        &-4\,x_{0}x_{3}x_{4}x_{5}x_{6}-4\,x_{0}x_{4}^{2}x_{5}x_{6}+2\,x_{0}^{2}x_{5}^{2}x_{6}+2\,x_{0}x_{3}x_{5}^{2}x_{6}+2\,x_{0}x_{5}^{3}x_{6}+x_{0}^{3}x_{6}^{2}+3\,x_{0}^{2}x_{3}x_{6}^{2} \\
        &+3\,x_{0}x_{3}^{2}x_{6}^{2}+x_{3}^{3}x_{6}^{2}x_{0}^{2}x_{4}x_{6}^{2}+2\,x_{0}x_{3}x_{4}x_{6}^{2}+3\,x_{3}^{2}x_{4}x_{6}^{2}\,x_{0}x_{4}^{2}x_{6}^{2}+3\,x_{3}x_{4}^{2}x_{6}^{2}+x_{4}^{3}x_{6}^{2}+3\,x_{0}^{2}x_{5}x_{6}^{2} \\
        &+6\,x_{0}x_{3}x_{5}x_{6}^{2}+3\,x_{3}^{2}x_{5}x_{6}^{2}+2\,x_{0}x_{4}x_{5}x_{6}^{2}+6\,x_{3}x_{4}x_{5}x_{6}^{2}+3\,x_{4}^{2}x_{5}x_{6}^{2}+4\,x_{0}x_{5}^{2}x_{6}^{2}+3\,x_{3}x_{5}^{2}x_{6}^{2} \\
        &+3\,x_{4}x_{5}^{2}x_{6}^{2}+x_{5}^{3}x_{6}^{2}+3\,x_{0}^{2}x_{6}^{3}+6\,x_{0}x_{3}x_{6}^{3}+3\,x_{3}^{2}x_{6}^{3}+4\,x_{0}x_{4}x_{6}^{3}+6\,x_{3}x_{4}x_{6}^{3}+3\,x_{4}^{2}x_{6}^{3}+6\,x_{0}x_{5}x_{6}^{3} \\
        &+6\,x_{3}x_{5}x_{6}^{3}+6\,x_{4}x_{5}x_{6}^{3}+3\,x_{5}^{2}x_{6}^{3}+3\,x_{0}x_{6}^{4}+3\,x_{3}x_{6}^{4}+3\,x_{4}x_{6}^{4}+3\,x_{5}x_{6}^{4}+x_{6}^{5}\rangle.\\
    \end{align*}
    $Y(2,3)^\ast = Y(r - 2, d - r + 2)^\ast$ for $r = 4$, $d = 5$, so by \Cref{thm: homaloidal 1} $Y(2,3)^\ast$ is homaloidal. In step 10 of our code, we confirm that the polynomial above is homaloidal by inverting the Jacobian map using the \texttt{inverseOfMap} command in the \texttt{RationalMaps} package \cite{RationalMapsSource, RationalMapsArticle}.
\end{example}

\subsection{Products of Homaloidal Polynomials}

We end this section by providing a general result on constructing homaloidal polynomials. More specifically, we show that the product of two homaloidal polynomials is again homaloidal when the two polynomials use no common variables.

\begin{lemma}
\label{lemma:product of homaloidals}
If $f(x) \in \mathbb{C}[x_0,x_1,\ldots,x_n]$ and $g(y)\in \mathbb{C}[y_0,y_1,\ldots,y_k]$ are two homaloidal polynomials in different polynomial rings (with no common variables), then their product $f(x)g(y)\in \mathbb{C}[x_1,\ldots,x_n,y_1,\ldots,y_k]$ is also homaloidal.  \end{lemma}
\begin{proof}
Let $\mathcal{L}_1 = \mathbb{C}^{n+1}$ and $\mathcal{L}_2 = \mathbb{C}^{k+1}$, respectively. 
By Theorem 13 of \cite{SDRshortProof2021}, we know that there exist sets of symmetric matrices $A_0,A_1,\ldots, A_n$ and $B_0,B_1,\ldots, B_k$ such that
\begin{eqnarray*}
f(x)= \det A(x), &\text{ where } & A(x):= A_0 + \sum_{i=1}^n x_iA_i, \text{ and } \\
g(y)= \det B(x), &\text{ where } & B(y):= B_0 + \sum_{i=1}^k y_iB_i.
\end{eqnarray*}
Let $A_1,A_2,\ldots, A_n$ (and similarly $B_1,B_2,\ldots, B_k$) be linearly independent. Then using Proposition 2.3 \cite{améndola2023differential} we can construct two affine linear embeddings $\mathcal{A}$ and $\mathcal{B}$, which gives us the following equations:
\[
\mld_{\det}(\mathcal{A}(\mathcal{L}_1))=\mld_f(\mathcal{L}_1) \text{ and } \mld_g(\mathcal{L}_2)=\mld_{\det}(\mathcal{B}(\mathcal{L}_2)).
\]
Further, as both $f$ and $g$ are homaloidal, $\mld_f(\mathcal{L}_1)$ and $\mld_g(\mathcal{L}_2)$ are equal to $1$ by Proposition 4.2 b \cite{améndola2023differential}. 

Now, let $\mathcal{L}=\mathcal{L}_1 \oplus \mathcal{L}_2$, 
and $C(x,y)$ be the block diagonal matrix $\mbox{diag}(A(x),B(y))$. We can construct a similar affine linear embedding $\mathcal{C}$ using the fact that the  block diagonal matrices $\{\mbox{diag}(A_i,0) \mid i=1,2,\ldots,n \} \cup \{ \mbox{diag}(0,B_j) \mid j=1,2,\ldots k\}$ are also linearly independent. As $1=\mld_{\det}(\mathcal{C}(\mathcal{L}))=\mld_{fg}(\mathcal{L})$, we conclude that $f(x)g(y)$ is also homaloidal.

In the case when either $A_1, A_2,\ldots, A_n$ or $B_1,B_2,\ldots, B_k$ (or both) are linearly dependent, we use the similar construction as in the proof of Proposition 2.3 in \cite{améndola2023differential} to show that 
\[
\mld_{\det}(\mathcal{A}(\mathcal{L}_1))=\mld_{f^2}(\mathcal{L}_1) \text{ and } \mld_{\det}(\mathcal{B}(\mathcal{L}_2))=\mld_{g^2}(\mathcal{L}_2),
\]
implying that $\mld_{\det}(\mathcal{C}(\mathcal{L}))=\mld_{f^2g^2}(\mathcal{L})=1$. But as $\mld_{f^2g^2}(\mathcal{L})$ is equal to $\mld_{fg}(\mathcal{L})$, we can again conclude that $f(x)g(y)$ is homaloidal.
\end{proof}

\section{Symmetric Determinantal Representation for Homaloidal Polynomials} 
\label{sec: SDR for homaloidals}
Having constructed homaloidal polynomials in \Cref{sec: examples of homaloidals}, the proof of \Cref{prop: homaloidal to det} shows that in order to find the corresponding ML degree one linear concentration models, one needs to express these polynomials as the determinants of matrices with linear entries. This section studies such representations, which are known as symmetric determinantal representations (SDRs); see \Cref{def: coordinate-free ML degree}.

\subsection{Bounds on SDR size}
While it is known that any polynomial $p\in \mathbb{F}[x_1,\ldots,x_n]$ admits a symmetric determinantal representation when $\mathbb{F}$ is a field with characteristic not equal to 2~\cite{SDRshortProof2021}, this subsection provides an upper bound on the size of the matrices in the SDR of a given polynomial when $\mathbb{F}\in \{\mathbb{R},\mathbb{C}\}$.

\begin{definition}
Let $u,v,w$ be three (not necessarily distinct) variables. We call the substitution $u\to v^2$ a \textit{square substitution}, and the substitution $u\to v\cdot w$ a \textit{product substitution}. A \textit{simple substitution} refers to a square or product substitution. \end{definition}
\begin{prop} \label{prop: simple substitutions}
Let $p \coloneqq y^d$ be a polynomial in $\mathbb{F}[y]$. Then up to a renaming of variables, $p$ can be obtained from one variable after $m\coloneqq \lfloor \log_2 d \rfloor + 1$ rounds of simple substitutions such that 
\begin{enumerate}[i.]
    \item there are at most $i$ simple substitutions in round i, where $i\in [m]$; and
    \item there is at most $1$ product substitution in each round.
\end{enumerate}
\end{prop}
\begin{proof}
    Let $d = \sum_{\ell = 1}^M 2^{p_{\ell}}$, where $p_1,\ldots,p_M$ are non-negative integers and $m -1= p_1 > p_2 >\cdots > p_M$.
    Let the initial variable be $u_1$. For $i\in [m]$, define the simple substitution(s) in round $i$ as follows:
    \begin{align} \label{eq: simple substitutions}
        u_j \to \begin{cases}  
        u_j \cdot u_{j+1} & \text{if $j=i$ and $M \geq j+1$}\\
        u_j^2 & \text{if $j = M$ and $p_M \geq i-M+1$} \\
        u_j^2 & \text{if $j< \min \{i,M\}$ and $p_j \geq i-j$} 
        \end{cases}
        \ \text{for all $j\in [\min \{i,M\}]$.}
    \end{align}
    By induction on $i$, we will show that for $i\in [m]$, after round $i$, we get 
    \begin{align} \label{eq: substitutions induction}
        \prod_{j=1}^{k_i} u_j^{2^{\min \{ i-j,p_j  \}}} \cdot u_{k_i+1}^{2^{\min\{i-k_i,p_{k_i+1}\}}},
    \end{align}
    where $k_i=\min\{i,M-1\}$. Then given that $M\leq m$, we can confirm after the $m$th round, we get 
    $
        \prod_{j=1}^{M-1} u_j^{2^{p_j}}\cdot u_M^{2^{p_M}}.
    $
    Eventually, renaming all the variables $u_1,\ldots,u_M$ to $y$ results in $y^{2^{p_1}+\cdots +2^{p_{M-1}}+2^{p_M}} = y^d = p$.
    
    For proving \eqref{eq: substitutions induction}, first note that after round $1$, if $M\geq 2$, we get $u_1\cdot u_2$, if $M=1$ and $p_1\geq 1$, we get $u_1^2$, and if $M=1$ and $p_1=0$, we get $u_1$. Now assume that for some $i\in [m-1]$, after round $i$, we do get \eqref{eq: substitutions induction}. 
    
    For all $j\in [k_i]$, we have $j< k_i+1 = \min\{i+1,M\}$. So, by \eqref{eq: simple substitutions}, in round $i+1$, the simple substitution $u_j \to u_j^2$ happens only if $p_j\geq i+1-j$. Thus, one of the following happens:
    \begin{enumerate}[i.]
        \item If $p_j\geq i+1-j$, we have $u_j^{2^{\min \{ i-j,p_j  \}}} = u_j^{i-j}$, and after round $i+1$, it changes to 
        $\left(u_j^2\right)^{2^{i-j}} = u_j^{2^{i+1-j}} = u_j^{2^{\min \{ i+1-j,p_j  \}}}$.
        \item If $p_j < i+1-j$, then no substitution happens for $u_j$ in round $i+1$, and we have 
    $u_j^{2^{\min \{ i-j,p_j  \}}} = u_j^{2^{p_j}} = u_j^{2^{\min \{ i+1-j,p_j  \}}}$.
    \end{enumerate}
     So, in any case, the first part of \eqref{eq: substitutions induction} changes to $\prod_{j=1}^{k_i} u_j^{2^{\min \{ i+1-j,p_j  \}}}$ after round $i+1$.
    
    Moreover, one of the following happens to $u_{k_i+1}$ in round $i+1$:
    \begin{enumerate}[i.]
        \item If $M\geq i+2$, then $k_{i+1}=i+1=k_i+1$. So, according to \eqref{eq: simple substitutions}, the simple substitution $u_{k_i+1} \to u_{k_i+1} \cdot u_{k_i+2}$ happens in round $i+1$, which means that $u_{k_i+1}^{2^{\min\{i-k_i,p_{k_i+1}\}}}=u_{k_i+1}$ changes to 
        \begin{align*}
            u_{k_i+1}\cdot u_{k_i+2} = u_{k_{i+1}}^{2^{\min\{i+1-k_{i+1},p_{k_{i+1}}\}}} \cdot u_{k_{i+1}+1}^{2^{\min\{i+1-k_{i+1},p_{k_{i+1}+1}\}}}.
        \end{align*}
        This proves the claim in this case.
        \item If $M \leq i+1$, then $k_{i+1}=k_i=M-1$. So, by \eqref{eq: simple substitutions}, only if $p_M \geq i-M+2$, the substitution $u_{k_i+1} \to u_{k_i+1}^2$ happens in round $i+1$. Therefore, if $p_M \geq i-M+2$, then in round $i+1$, $u_{k_i+1}^{2^{\min\{i-k_i,p_{k_i+1}\}}}= u_{k_i+1}^{2^{i-M+1}}$ changes to
        \begin{align*}
          \left(  u_{k_i+1}^2\right)^{2^{i-M+1}} =  u_{k_{i+1}+1} ^{2^{\min\{i+1-k_{i+1},p_{k_{i+1}+1}\}}}.
        \end{align*}
        On the other hand, if $p_M \leq i-M+1$, then no substitution happens for $u_{k_i+1}$ in round $i+1$, and we have
        \begin{align*}
            u_{k_i+1}^{2^{\min\{i-k_i,p_{k_i+1}\}}} = u_{M} ^ {p_M} = u_{k_{i+1}+1} ^{2^{\min\{i+1-k_{i+1},p_{k_{i+1}+1}\}}}.
        \end{align*}
        In either case, the claim is proved.
    \end{enumerate}
    This concludes the proof of \eqref{eq: substitutions induction}, and hence, the proposition.
\end{proof}

\begin{remark} \label{remark: simple substitutions}
    According to the proof of \Cref{prop: simple substitutions}, if $d=2^{m-1}$ for some $m\in \mathbb{N}$, then $p=y^d$ can be obtained from one variable after $m-1$ rounds of simple substitutions, each consisting of merely one square substitution.
\end{remark}
The proof of the following proposition is motivated by the construction proposed in~\cite{SDRshortProof2021}.
\begin{prop} \label{prop: det rep}
    Assume the polynomial $q(x_1,\ldots,x_r,u_1,\ldots,u_s,v_1,\ldots,v_t,w_1,\ldots,w_t)$ is obtained from the polynomial $p(x_1,\ldots,x_r,y_1,\ldots,y_s,z_1,\ldots,z_t)$ after the set of square substitutions $\{y_1 \to u_1^2, \ldots, y_s \to u_s^2\}$ and the set of product substitutions $\{z_1 \to v_1\cdot w_1, \ldots, z_t \to v_t \cdot w_t\}$. If $p$ has a determinantal representation consisting of symmetric matrices of size $k$, then $q$ has a determinantal representation consisting of symmetric matrices of size at most $k (1+2s+4t) + 1$.
\end{prop}

\begin{proof}
Let the determinantal representation of $p$ be as follows:
\begin{align*}
    p(x_1,\ldots,x_r,y_1,\ldots,y_s,z_1,\ldots,z_t) = \det\left(A_0 + \sum_{i=1}^r x_i A_i + \sum_{i=1}^ s y_i B_i + \sum_{i=1}^ t z_i C_i\right),
\end{align*}
where $A_0,\ldots,A_r,B_1,\ldots,B_s,C_1,\ldots,C_t$ are constant symmetric $k\times k$ matrices. Define
\begin{align*}
    M \coloneqq \begin{pmatrix}
        A_0 + \sum_{i=1}^r x_i A_i &  X^T  & Y^T\\
        X & \diag D_1 & 0\\
        Y & 0 & \diag D_2
    \end{pmatrix},
\end{align*}
where
\begin{align*}
    &X \coloneqq \begin{pmatrix}
        u_1 I_k \\ u_1 I_k \\ \vdots \\ u_s I_k \\ u_s I_k 
    \end{pmatrix},
    &&Y \coloneqq \begin{pmatrix}
        \frac12(v_1+w_1) I_k \\ \frac12(v_1-w_1) I_k \\ \frac12(v_1+w_1) I_k \\ \frac12(v_1-w_1) I_k \\ \vdots \\ \frac12(v_t+w_t) I_k \\ \frac12(v_t-w_t) I_k \\ \frac12(v_t+w_t) I_k \\ \frac12(v_t-w_t) I_k
    \end{pmatrix},
    &&D_1 \coloneqq \begin{pmatrix}
          -(B_1-\lambda_1I_k)^{-1}\\
          -\lambda_1^{-1}I_k\\
          \vdots  \\
          -(B_s-\lambda_sI_k)^{-1}\\
          -\lambda_s^{-1}I_k
    \end{pmatrix},
    &&D_2 \coloneqq \begin{pmatrix}
        -(C_1-\gamma_1I_k)^{-1} \\
        (C_1-\gamma_1I_k)^{-1} \\
        -\gamma_1^{-1}I_k  \\
         \gamma_1^{-1}I_k  \\
        \vdots \\
        -(C_t-\gamma_tI_k)^{-1} \\
         (C_t-\gamma_tI_k)^{-1}  \\
         -\gamma_t^{-1}I_k \\
         \gamma_t^{-1} I_k
    \end{pmatrix}
\end{align*}
such that for all $i\in [s]$, $\lambda_i\in \mathbb{R}\setminus\{0\}$ is not an eigenvalue of $B_i$, and for all $i\in [t]$, $\gamma_i \in \mathbb{R}\setminus \{0\}$ is not an eigenvalue of $C_i$. So, for all $i\in [s]$, $B_i - \lambda_i I_k$, and for all $i\in [t]$, $C_i-\gamma_i I_k$ are invertible matrices. 

Then the Schur complement of $M$ with respect to the block 
$\tilde M \coloneqq 
\begin{pmatrix}
    \diag D_1 & 0 \\ 0 & \diag D_2
\end{pmatrix}
$
is 
\begin{align*}
\frac{M}{\tilde{M}} =& A_0 + \sum_{i=1}^r x_i A_i - \begin{pmatrix} X^T & Y^T \end{pmatrix}\cdot (\tilde{M})^{-1} \cdot \begin{pmatrix} X \\ Y \end{pmatrix} \\
                    =& A_0 + \sum_{i=1}^r x_i A_i - X^T\left(\diag D_1\right)^{-1} X - Y^T\left(\diag D_2\right)^{-1} Y\\
                    =& A_0 + \sum_{i=1}^r x_i A_i - \sum_{i=1}^s u_i^2 (-(B_i - \lambda_i I_k)) - \sum_{i=1}^s u_i^2 (-\lambda_iI_k)\\
                    -&\sum_{i=1}^t \frac14 (v_i+w_i)^2 (-(C_i-\gamma_i I_k))  -\sum_{i=1}^t \frac14 (v_i-w_i)^2 (C_i-\gamma_i I_k) \\
                    -&\sum_{i=1}^t \frac14 (v_i+w_i)^2 (-\gamma_iI_k) - \sum_{i=1}^t \frac14 (v_i-w_i)^2 (\gamma_iI_k)\\
                    =& A_0 + \sum_{i=1}^r x_i A_i +\sum_{i=1}^s u_i^2 B_i + \sum_{i=1}^t \frac14 (v_i+w_i)^2 C_i -\sum_{i=1}^t \frac14 (v_i-w_i)^2 C_i \\
                    =&  A_0 + \sum_{i=1}^r x_i A_i +\sum_{i=1}^s u_i^2 B_i + \sum_{i=1}^t v_i w_i C_i.
\end{align*}
Therefore,
\begin{align*}
    \det(M) = \det(\tilde{M})\cdot \det\left(\frac{M}{\tilde{M}}\right) = \det(\tilde{M}) \cdot p|_{y_i \to u_i ^2 \text{ for all }i \in [s],\ z_i \to v_iw_i \text{ for all } i\in [t]} = \det(\tilde{M}) \cdot q,
\end{align*}
which means that 
\begin{align*}
    q = \frac{1}{\det(\tilde{M})} \cdot \det(M) = \det \begin{pmatrix}M & 0 \\ 0 & (\det(\tilde{M}))^{-1} \end{pmatrix}.
\end{align*}
This is a symmetric determinantal representation for $q$ given that $\tilde{M}$ is a constant matrix. Note that $M$ is a matrix of size $k(1+2s+4t)$, so, we have expressed $q$ as the determinant of a symmetric matrix of size $k(1+2s+4t)+1$.
\end{proof}
\begin{remark} \label{remark: det rep}
In the construction outlined in the proof of \Cref{prop: det rep}, for any $i\in [s]$, if $B_i$ is an invertible matrix, $\lambda_i$ can be chosen to be $0$ and the rows related to $-\lambda_i^{-1}I_k$ in $D_1$ along with the corresponding rows in $X$ can be omitted. Similarly, for any $i\in [t]$, if $C_i$ is an invertible matrix, $\gamma_i$ can be chosen to be 0 and the rows related to $-\gamma_i^{-1} I_k$ and $\gamma_i^{-1} I_k$ in $D_2$ along with the corresponding rows in $Y$ can be removed. Hence, if all of the matrices $B_1,\ldots,B_s$ and $C_1,\ldots,C_t$ are invertible, then $q$ can be expressed as the determinant of a symmetric matrix of size $k(1+s+2t)+1$.

Moreover, in the last line of the proof, $q=\frac{1}{\det(\tilde{M})} \cdot \det(M)$ can be also viewed as 
\begin{align}\label{eq alternative matrix for det rep}
    \det \left( (\det (\tilde{M}))^{\frac{-1}{m}} \cdot M \right),
\end{align}
where $m$ is the size of matrix $M$. The determinantal representation \eqref{eq alternative matrix for det rep} has been avoided in the original proof due to the fact that in some cases, $(\det (\tilde{M}))^{\frac{-1}{m}}$ can be a non-real complex number. The construction proposed in the proof, however, is guaranteed to express $q$ as the determinant of a real symmetric matrix if $p$ can be written as the determinant of a real symmetric matrix.
\end{remark}
\begin{prop} \label{prop: expressing a poly of degree d as a sum of degree d's}
    Suppose $\mathbb{F}$ is either the field of real numbers or the field of complex numbers. Let $p\in \mathbb{F}[x_1,\ldots,x_n]$ be a polynomial of degree at most $d$ defined as follows:
    \begin{align} \label{eq: degree d polynomial}
        p(x_1,\ldots,x_n) = \sum_{\substack{i_1,\ldots,i_n\in \mathbb{Z}_{\geq 0} \\ \vspace{0.5mm} \\ i_1+\cdots+i_n\leq d}} \alpha_{i_1,\ldots,i_n} \ x_1^{i_1} \cdots x_n^{i_n},
    \end{align}
    where $\alpha_{i_1,\ldots,i_n}\in \mathbb{F}$ for all tuples $(i_1,\ldots,i_n)\in \mathbb{Z}^n_{\geq 0}$ with $i_1+\cdots+i_n\leq d$. Define the tensor $T\in S^d(\mathbb{F}^{n+1})$ such that
    \begin{align}\label{eq: coeff tensor}
        T_{\underbrace{1\ldots 1}_{i_1}\ \ldots \ \underbrace{n\ldots n}_{i_n} \ \underbrace{(n+1)\ldots (n+1)}_{i_{n+1}}}\coloneqq \frac{\alpha_{i_1,\ldots,i_n}}{\binom{d}{i_1,\ldots,i_{n+1}}}
    \end{align}
    for all $i_1,\ldots,i_{n+1}\in \mathbb{Z}_{\geq 0}$ with $i_1+\cdots+i_{n+1}=d$. 
    If $T$ has the symmetric decomposition $T = \sum_{i=1}^r \lambda_i w_i^{\otimes d}$, where $\lambda_1,\ldots,\lambda_r\in \mathbb{F}$ and $w_1,\ldots,w_r\in \mathbb{F}^{n+1}$, then  $p$ can be expressed as $p=\sum_{i=1}^r \lambda_i f_i^d$, where for all $i\in [r]$, $f_i\in \mathbb{F}[x_1,\ldots,x_n]$ is an affine polynomial.
\end{prop}
\begin{proof}
    Define $\mathbf{x}\coloneqq  [\begin{array}{cccc} x_1 & \cdots & x_n & 1\end{array}]$, and denote the contraction of any tensor $S$ with $\mathbf{x}$ along the $i$th way of $S$ by $S\bullet_i \mathbf{x}$. The contraction of $T$ with vector $\mathbf{x}$ along each of the ways of $T$ gives us
\begin{align}
   \contraction =& \sum_{\ell_1=1}^{n+1} \ldots \sum_{\ell_d=1}^{n+1} \mathbf{x}_{\ell_1} \cdots \mathbf{x}_{\ell_d} \ T_{\ell_1\ldots\ell_d} \notag\\
                                                =&\sum_{\substack{i_1,\ldots,i_n,i_{n+1}\in \mathbb{Z}_{\geq 0} \\ \vspace{0.5mm} \notag \\ i_1+\cdots+i_n+i_{n+1}=d}}  \binom{d}{i_1,\ldots,i_{n+1}} \ T_{\underbrace{1\ldots 1}_{i_1}\ \ldots \ \underbrace{n\ldots n}_{i_n} \ \underbrace{(n+1)\ldots (n+1)}_{i_{n+1}}} \  \mathbf{x}_1^{i_1}\cdots \mathbf{x}_n^{i_n} \ \mathbf{x}_{n+1}^{i_{n+1}} \notag \\
                                                =&\sum_{\substack{i_1,\ldots,i_n\in \mathbb{Z}_{\geq 0} \\ \vspace{0.5mm} \\ i_1+\cdots+i_n\leq d}}  \alpha_{i_1,\ldots,i_n} x_1^{i_1}\cdots x_n^{i_n} 
                                                = p(x_1,\ldots,x_n).  \notag
\end{align}
Now given the symmetric decomposition of $T$,  
\begin{align*}
    &p(x_1,\ldots,x_n) = \contraction\\
    =& \sum_{\ell_1=1}^{n+1} \ldots \sum_{\ell_d=1}^{n+1} \mathbf{x}_{\ell_1} \cdots \mathbf{x}_{\ell_d} \ T_{\ell_1\ldots\ell_d}
    =\sum_{\ell_1=1}^{n+1} \ldots \sum_{\ell_d=1}^{n+1} \mathbf{x}_{\ell_1} \cdots \mathbf{x}_{\ell_d} \left(\sum_{i=1}^r \lambda_i (w_i)_{\ell_1}\cdots(w_i)_{\ell_{d}}\right)\\
    =&\sum_{i=1}^r \lambda_i \left( \sum_{\ell=1}^{n+1} \mathbf{x}_\ell\  (w_{i})_\ell \right)^d 
    = \sum_{i=1}^r \lambda_i \langle \mathbf{x},w_i \rangle^d.
\end{align*}
For each $i\in [r]$, define $f_i \coloneqq \langle  \mathbf{x},w_i\rangle$, and the proof is concluded.
\end{proof}

\begin{theorem}\label{thm: SDR upperbound}
Assume $\mathbb{F}$ is either the field of real numbers or the field of complex numbers, and  $p\in \mathbb{F}[x_1,\ldots,x_n]$ is a polynomial of degree at most $d$ as described in \eqref{eq: degree d polynomial}. Let $r$ be the symmetric rank of tensor \eqref{eq: coeff tensor} over $\mathbb{F}$. Then the following hold:
\begin{enumerate}[i.]
    \item If $m\coloneqq \lfloor \log_2 d \rfloor + 1$, then $p$ has a symmetric determinantal representation with matrices  of size at most $2^{m-1} r^{m} (m+2)!$ with entries in $\mathbb{F}$.
    \item If $d=2^{m-1}$, then $p$ has a symmetric determinantal representation with matrices of size at most $(2r+2)^{m-1}$ with entries in $\mathbb{F}$.
\end{enumerate}
\end{theorem}
\begin{proof}
    In either of the above cases, the proof is concluded if $\sum_{i=1}^r \lambda_i y_i^d\in \mathbb{F}[y_1,\ldots,y_r]$ has a symmetric determinantal representation of the desired size. This is because each $y_i$, $i\in [r]$, can be replaced by $f_i$ in the determinantal representation, where $f_i$ is the affine polynomial mentioned in \Cref{prop: expressing a poly of degree d as a sum of degree d's}. Note that the coefficients $\lambda_i$ can also be considered to be those mentioned in \Cref{prop: expressing a poly of degree d as a sum of degree d's}.

    By \Cref{prop: simple substitutions}, the polynomial $\sum_{i=1}^r \lambda_i y_i^d$ can be obtained from the polynomial $\sum_{i=1}^r \lambda_i u_i \in \mathbb{F}[u_1,\ldots,u_r]$ (up to a renaming of variables) after at most $m$ rounds of simple substitutions such that round $i$, $i\in [m]$, has at most $r\cdot i$ simple substitutions and there is at most $r$ product substitutions in each round. By \Cref{remark: simple substitutions}, if $d=2^{m-1}$, then the polynomial $\sum_{i=1}^r \lambda_i y_i^d$ can be obtained from the polynomial $\sum_{i=1}^r \lambda_i u_i$ (up to a renaming of variables) after $m-1$ rounds of simple substitutions each involving only $r$ square substitutions. 

    If the polynomial before round $i$ has a symmetric determinantal representation of size $k$, then by \Cref{prop: det rep}, for arbitrary $d$, after round $i$, one gets a representation of size at most $k(1+2r(i-1)+4r)+1$ for the new polynomial. Note that $k(1+2r(i-1)+4r)+1\leq k(2r(i+2))$. Moreover, if $d=2^{m-1}$, then the determinantal representation of the new polynomial has size at most $k(1+2r)+1$ after round $i$. Note that $k(1+2r)+1\leq k(2+2r)$.

    Given that 
    \begin{align*}
        \sum_{i=1}^r \lambda_i u_i = \det\left(\sum_{i=1}^r u_i [\lambda_i] \right),
    \end{align*}
    the initial polynomial has a determinantal representation of size 1. Therefore, for arbitrary $d$, the polynomial $p$ has a determinantal representation of size at most $2r(m+2)\cdot 2r(m+1) \cdots 2r(3) = 2^{m-1} r^m (m+2)!$. Furthermore, if $d=2^{m-1}$, then $p$ has a determinantal representation of size at most $(2r+2)^{m-1}$ as desired.
\end{proof}

\begin{cor}\label{cor: quadratic SDR}
    Assume $\mathbb{F}$ is either the field of real numbers or the field of complex numbers, and  $q\in \mathbb{F}[x_1,\ldots,x_n]$ is a polynomial of degree 2. Then for some $r\in [n+1]$, the polynomial $q$ can be expressed as the determinant of a symmetric matrix in $\mathbb{F}^{(r+2)\times (r+2)}$, whose diagonal entries are non-zero. 
\end{cor}
\begin{proof}
    Let $p=q-1$ be the polynomial \eqref{eq: degree d polynomial} with $d=2$. Then tensor \eqref{eq: coeff tensor} is actually an $(n+1)\times (n+1)$ matrix, and therefore, its symmetric rank $r$ is equal to its rank, which is at most $n+1$. Additionally, in the expression $p=\sum_{i=1}^r \lambda_i f_i^2$ described in \Cref{prop: expressing a poly of degree d as a sum of degree d's}, the coefficients $\lambda_i$ are in fact, the eigenvalues of this matrix.

    Since $p=\sum_{i=1}^r \lambda_i f_i^2$ is obtained from the linear polynomial $\sum_{i=1}^r\lambda_iu_i$ after 1 round of square substitutions $\{u_1\to f_1^2,\ldots,u_r\to f_r^2\}$, by \Cref{prop: det rep} and \Cref{remark: det rep}, 
    \begin{align*}
        p = \det\begin{pmatrix}
            0 & f_1 & \cdots & f_r & 0\\
            f_1 & -\lambda_1^{-1} &&&\\
            \vdots && \ddots &&\\
            f_r &&& -\lambda_r^{-1} &\\
            0 &&&& (-1)^r\lambda_1\cdots \lambda_r
        \end{pmatrix}.
    \end{align*}
    So, one can confirm
    \begin{align}\label{eq: quadric det}
    q = p+1 = 
         \det\begin{pmatrix}
            1 & f_1 & \cdots & f_r & 0\\
            f_1 & -\lambda_1^{-1} &&&\\
            \vdots && \ddots &&\\
            f_r &&&- \lambda_r^{-1} &\\
            0 &&&& (-1)^r\lambda_1\cdots \lambda_r
        \end{pmatrix}.
    \end{align}
\end{proof}
\begin{example}
    \label{ex: quadric sdr}
    Let $f(x,y,z) = \frac{7}{25} x^2 - y^2 - \frac{48}{25} xz - \frac{7}{25} z^2$. Then $f(x,y,z)$ has the following decomposition:
    \begin{align*}
        f(x,y,z) &= (x, y, z) 
        \begin{pmatrix}
            \frac{9}{25}&\frac{12}{25}&-\frac{4}{5}\\
            -\frac{4}{5}&\frac{3}{5}&0\\
            \frac{12}{25}&\frac{16}{25}&\frac{3}{5}
        \end{pmatrix}
        \begin{pmatrix}
            -1&0&0\\
            0&-1&0\\
            0&0&1
        \end{pmatrix}
        \begin{pmatrix}
            \frac{9}{25}&\frac{12}{25}&-\frac{4}{5}\\
            -\frac{4}{5}&\frac{3}{5}&0\\
            \frac{12}{25}&\frac{16}{25}&\frac{3}{5}
        \end{pmatrix}^\top
        \begin{pmatrix}
            x \\
            y \\
            z 
        \end{pmatrix} \\
        &= - \frac{1}{25^2} (9x - 20y + 12z)^2 - \frac{1}{25^2} (12x + 15y + 16z)^2 + \frac{1}{25^2} (-20x + 15z)^2.
    \end{align*}

    Following the construction above for the symmetric determinantal representation,
    \begin{equation}
        \label{eqn: quadric sdr}
        f(x,y,z) = \det \begin{pmatrix}
            0 & \frac{1}{25} (9x - 20y + 12z) & \frac{1}{25} (12x + 15y + 16z) & \frac{1}{25} (-20x + 15z) & \\
            \frac{1}{25} (9x - 20y + 12z) & 1 &&& \\
            \frac{1}{25} (12x + 15y + 16z) && 1 && \\
            \frac{1}{25} (-20x + 15z) &&& -1 & \\
            &&&& -1
        \end{pmatrix}.
    \end{equation}
\end{example}

\subsection{Positive Definiteness of the SDR Matrix}\label{subsec: SDR intersection}

An SDR of a polynomial $p\in \mathbb{R}[x_1,\ldots,x_n]$ expresses it as the determinant of a symmetric matrix $M(x_1,\ldots,x_n)$. In order for $M(x_1,\ldots,x_n)$ to be the concentration matrix of some Gaussian distribution, it needs to be positive (semi-)definite.  In this subsection, we make some observations on whether $M(x_1,\ldots,x_n)$ can be a positive (semi-)definite matrix on a non-empty subset of $\mathbb{R}^n$.
\Cref{lemma: pd diagonal} and \Cref{thm: sylvester's criterion} respectively provide a necessary condition and an equivalent condition for a matrix to be positive (semi-)definite.
\begin{lemma}
    \label{lemma: pd diagonal}
    Let $M$ be a positive (semi-)definite $n \times n$ matrix. Then the diagonal entries, $M_{ii}$, are positive (non-negative) for $i \in [n]$.
\end{lemma}

\begin{proof}
    By definition, $\bfx^\top M \bfx > 0$ ($\geq 0$) for all $\bfx \neq 0$. Let $\bfx = e_i$ be the $i$th standard basis vector. Then $e_i^\top M e_i = M_{ii} > 0$ ($\geq 0$).
\end{proof}

\begin{theorem}[Sylvester's Criterion]
    \label{thm: sylvester's criterion}
    A matrix $M$ is positive definite if and only if all leading principal minors of $M$ are positive.
\end{theorem}

Let $p\in \mathbb{R}[x_1,\ldots,x_n]$ be a polynomial. Our observations are the following:
\begin{enumerate}
    \item If the polynomial $p$ is negative on $\mathbb{R}^n$, then its SDR matrix $M(x_1,\ldots,x_n)$ can never be positive (semi-)definite since for all $(x_1,\ldots,x_n)\in \mathbb{R}^n$, we have $\det\left( M(x_1,\ldots,x_n)\right)=p(x_1,\ldots,x_n)<0$.
    \item Assume $p =q+c$, where $q$ is a 
    quadratic polynomial with no constant term, $c\in \mathbb{R}$, and $q\leq 0$ everywhere. If $p$ is sometimes positive on $\mathbb{R}^n$, then the SDR matrix of $p$ can be chosen such that its diagonal entries as well as its determinant are positive on a non-empty subset of $\mathbb{R}^n$. One can obtain an SDR matrix for $p$ by following the construction in \Cref{cor: quadratic SDR}. Note that in this case, for any $(x_1,\ldots,x_n)\in \mathbb{R}^n$, $q(x_1,\ldots,x_n) = [\begin{array}{cc} \mathbf{x} & 1\end{array}]\   T \  [\begin{array}{cc} \mathbf{x} & 1\end{array}]^T$, where $\mathbf{x} \coloneqq [\begin{array}{ccc} x_1 & \cdots & x_n \end{array}]$ and $T$ is the coefficient tensor (here matrix) of $q$ with the form
    \begin{align} \label{eq: coeff tensor for quadratics}
        T = \begin{pmatrix}
            M & \mathbf{b}\\
            \mathbf{b}^T & 0
        \end{pmatrix}
    \end{align}
    with $M\in \mathbb{R}^{n\times n}$ and $\mathbf{b}\in \mathbb{R}^n$.  We will show that all of the non-zero eigenvalues of $T$ are negative in this case, and therefore, all the diagonal entries in~\eqref{eq: quadric det} are positive. For eigenvalue $\lambda_i$, $i \in [r]$, assume $\mathbf{v_i}= [\begin{array}{cc} \mathbf{x} & x_{n+1}\end{array}]^T$ is its corresponding eigenvector. If $x_{n+1}$ can chosen to be 1, then we have
    \begin{align*}
        0\geq q(x_1,\ldots,x_n) = \mathbf{v_i}^T \  T \  \mathbf{v_i} = \lambda_i ||\mathbf{v_i}||^2,
    \end{align*}
    which implies that $\lambda_i\leq 0$. If $x_{n+1}=0$, then we have
    \begin{align*}
       \begin{pmatrix} M\mathbf{x}^T  \\ \mathbf{b}^T\mathbf{x}^T\end{pmatrix} =T \mathbf{v_i} = \lambda_i \mathbf{v_i} = \begin{pmatrix} \lambda_i\mathbf{x}^T \\ 0\end{pmatrix},
    \end{align*}
    which implies that
    \begin{align*}
        0\geq q(x_1,\ldots,x_n) = \begin{pmatrix} \mathbf{x} & 1\end{pmatrix} \ T\ \begin{pmatrix} \mathbf{x}^T \\ 1\end{pmatrix} = \mathbf{x}\ M \ \mathbf{x}^T + 2\ \mathbf{b}^T \mathbf{x}^T = \lambda_i ||\mathbf{x}||^2,
    \end{align*}
    and thus, $\lambda_i\leq 0$ in this case as well. 
\end{enumerate}

We end this subsection with our last observation, which we state as a proposition due to its significance:
\begin{prop}~\label{prop: degree not a power of 2}
    If $p$ is of degree $d$, where $d$ is not a power of 2, the construction in Theorem~\ref{thm: SDR upperbound} entails at least one constant as well as its negative on the diagonal of $M(x_1,\ldots,x_n)$. In this case, $M(x_1,\ldots,x_n)$ can never be a positive definite matrix. 
\end{prop}

Proposition~\ref{prop: degree not a power of 2} implies that the construction proposed in \Cref{thm: SDR upperbound} can only result in a potentially positive definite SDR matrix $M(x_1,\ldots,x_n)$ for a polynomial $p$ when the degree of $p$ is a power of 2.

\subsection*{Discussion}
In Section \ref{sec: spanning tree models}, we showed that the spanning tree generating function, $P_G$, of a graph $G$ is homaloidal if $G$ is chordal, and is not homaloidal when $G$ is a cycle. Computations further suggest that this polynomial is not homaloidal when $G$ is not a chordal graph in general. In \Cref{sec: examples of homaloidals} we studied homaloidal polynomials given by Dolgachev's classification and another construction. \Cref{sec: SDR for homaloidals} provides an algorithm to compute a relatively small SDR for any polynomial. However, these SDRs are affine linear and may not always intersect the cone of positive definite matrices. It would be interesting to see which of the polynomials from \Cref{sec: examples of homaloidals} arise as a polynomial $P_G$ for some graph $G$. 

\subsection*{Acknowledgements} Many thanks to our group mentor, Serkan Ho\c{s}ten, and to the other Apprenticeship Week organizers, Bernd Sturmfels and Kaie Kubjas. Special thanks also to Tobias Boege, Jane Coons, Igor Dolgachev, Joseph Johnson, Elina Robeva, Frank R\"ottger, and Piotr Zwiernik for helpful conversations. Part of this research was performed while the authors were visiting the Institute for Mathematical and Statistical Innovation (IMSI), which is supported by the National Science Foundation (Grant No. DMS-1929348). 
The first author was supported by the National Science Foundation Graduate Research Fellowship under Grant No. DGE-1841052, and by the National Science Foundation under Grant No. 1855135 during the writing of this paper.
The second author received funding from the European Research Council (ERC) under the European Union’s Horizon 2020 research and innovation programme (grant agreement No. 883818). 
The third author was supported by the NSERC Discovery Grant DGECR-2020-00338.

\bibliographystyle{alpha}
\bibliography{main.bib}

\end{document}